\documentclass[11pt]{amsart}
\usepackage[colorlinks=true, pdfstartview=FitV, linkcolor=blue, 
citecolor=blue]{hyperref}

\usepackage{amsmath,amssymb,bbm,amscd}
\usepackage{graphicx}
\usepackage{a4wide}
\usepackage{mathtools}
\usepackage{tikz-cd}

\theoremstyle{plain}
\newtheorem{theorem}{Theorem}[section]
\newtheorem{prop}[theorem]{Proposition}

\newtheorem{coro}[theorem]{Corollary}

\theoremstyle{definition}

\newtheorem{remark}[theorem]{Remark}
\newtheorem{definition}[theorem]{Definition}

\numberwithin{equation}{section}
 
\newcommand{\ts}{\hspace{0.5pt}}
\newcommand{\nts}{\hspace{-0.5pt}}

\DeclareMathOperator{\card}{\mathrm{card}}

\DeclareMathOperator{\vol}{\mathrm{vol}}

\newcommand{\cA}{\mathcal{A}}

\newcommand{\cL}{\mathcal{L}}
\newcommand{\vL}{\varLambda}

\newcommand{\ZZ}{\mathbb{Z}\ts}
\newcommand{\RR}{\mathbb{R}\ts}

\newcommand{\QQ}{\mathbb{Q}}

\newcommand{\one}{\mathbbm{1}}

\newcommand{\defeq}{\mathrel{\mathop:}=}
\newcommand{\eqdef}{=\mathrel{\mathop:}}

\newcommand{\oplam}{\mbox{\Large $\curlywedge$}} 
\newcommand{\dd}{\mathop{}\!\mathrm{d}}
\newcommand{\exend}{\hfill $\Diamond$}

\newcommand{\myfrac}[2]{\frac{\raisebox{-2pt}{$#1$}}
      {\raisebox{0.5pt}{$#2$}}}
\newcommand{\bs}[1]{\boldsymbol{#1}}

\newcommand{\cvg}{\mathrm{cvg}\ts}

\begin{document}

\title[Pair correlations and covariograms of Rauzy fractals]{Pair
  correlations of one-dimensional model sets \\[3mm] and monstrous
  covariograms of Rauzy fractals}

\author{Michael Baake}

\author{Anna Klick}

\author{Jan Maz\'{a}\v{c}}

\address{Fakult\"{a}t f\"{u}r Mathematik,
  Universit\"{a}t Bielefeld,\newline \hspace*{\parindent}Postfach
  100131, 33501 Bielefeld, Germany}  
\email{$\{$mbaake,aklick,jmazac$\}$@math.uni-bielefeld.de}

\makeatletter
\@namedef{subjclassname@2020}{%
  \textup{2020} Mathematics Subject Classification}
\makeatother



\begin{abstract}
  The averaged distance structure of one-dimensional regular model
  sets is determined via their pair correlation functions. The latter
  lead to covariograms and cross covariograms of the windows, which
  give continuous functions in internal space. While they are simple
  tent-shaped, piecewise linear functions for intervals, the typical
  case for inflation systems leads to convolutions of Rauzy fractals,
  which are difficult to compute. In the presence of an inflation
  structure, an alternative path is possible via the exact
  renormalisation structures of the pair correlation functions. We
  introduce this approach and derive two concrete examples, which
  display an unexpectedly complex and wild behaviour.
\end{abstract}

\maketitle
\thispagestyle{empty}

\section{Introduction}\label{sec:intro}
The distance structure of lattices and periodic point sets is
efficiently summarised in the corresponding theta series, which are
well studied and understood, and have close connections with modular
forms \cite{ConSl}. Much less is known for aperiodic point sets, even
for the most popular ones such as the Fibonacci or the silver mean
chain. It is questionable to what extent individual theta series are
meaningful, as no two members of the Fibonacci hull would have the
same series, and similarly for all other (regular) model sets.

A natural question would then be what the \emph{averaged} distance
structure is, where the average is taken over all elements of the hull
with a point at 0. This average is well defined and can often be
determined, at least in principle, from one element of the hull. For
regular model sets, which are particularly nice cut-and-project sets,
the answer is then given in terms of (sums of) pair correlation
functions. The latter, in turn, lead to covariograms of the coding
windows, or cross covariograms between different windows. These are
always continuous functions in internal space, which sounds nice and
easy. If the windows are simple polygons, one indeed obtains decent
functions. Nevertheless, no clear relation to aperiodic analogues of
theta series has materialised so far, despite several efforts in this
direction \cite{AKM2024, NOS, BHK,MW}.

In view of this, potential insight might come from investigating some
less prominent examples that are nevertheless somehow typical. Indeed,
even within the class of one-dimensional inflation tilings, windows of
Rauzy fractals are ubiquitous, but have not been studied from this
point of view. One reason certainly is the obstacle of computing the
covariogram of a~Rauzy fractal, which already in one dimension tends
to be a Cantorval \cite{BGM2024}.  The covariograms are continuous
functions (being convolutions of two functions that are both $L^{1}$
and $L^{\infty}$), with a~behaviour that is somewhat reminiscent to
that of a Weierstrass monster.

In fact, until recently, hardly any reliable technique for computing
these functions was known. Within the class of inflation point sets,
the identification of an exact renormalisation scheme for the pair
correlation functions in \cite{BGMan,ManiboThesis} opens the door to
some efficient and exact calculations, as we will demonstrate by way
of two examples (for more, see \cite{Klick}). The crucial connection
comes from the uniform distribution property of the model set points
within the windows, after their lift via the $\star$-map. Consequently,
the renormalisation scheme permits the calculation of the (cross)
covariograms at a dense set of points, which is enough for the
illustration of a continuous function.

This looks like a maybe tedious, but ultimately easy thing to do ---
until one actually does it.  Indeed, the spiky behaviour of the
covariograms turns out to be unexpectedly wild, and the convergence of
the approximants is surprisingly slow --- at least in some
examples. It is thus the goal of this paper to introduce the problem
and method via some characteristic examples, and to demonstrate that
this is a fractal phenomenon worth further study.

This paper is organised as follows. In Section \ref{section2}, we set
the scene and introduce the key concepts and notions around
covariograms, cut-and-project sets and pair correlation functions. In
the main Section \ref{section3}, we then derive the results for two
binary inflation systems. The first is based on the silver mean
inflation factor, $\lambda = 1+\sqrt{2}$, and the second on its
square. A central block will be the explicit derivation of the exact
renormalisation equations, and how they are used to approach the
(cross) covariograms.

\section{Preliminaries}\label{section2}

Let us begin by recalling some concepts and notions; see \cite{TAO}
for general background.

\begin{definition}\label{def:cvg}
   Let $W \subset \RR^{m}$ be a non-empty compact set. We define the 
   \emph{covariogram} of~$W$ for almost all $x \in \RR^{m}$ as
\[
   \cvg^{}_{W}(x)\, \defeq \, \vol\ts\bigl(W \cap (x+W)\bigr) \, = \, 
   \big( \mathbf{1}^{}_{-W}  \ast \mathbf{1}^{}_{W}\big)(x)\ts,
\]
where $\mathbf{1}^{}_{W}$ is the characteristic function of
$W$\hspace{-.06cm}, and $\ast$ stands for the usual convolution of two
functions $f,\ts g\in L^{1}(\RR^m)$ as given by
\[
     (f\ast g)(x) \, = \int_{\RR^m} f(y)\ts g(x-y) \dd y \ts ,
\]
with the usual understanding that this is well defined for almost all
$x\in\RR^m$.
\end{definition}

Arguably, this is most famously known in relation to the
\textit{covariogram problem} whether the covariogram of a convex body
in $\RR^m$ determines the body uniquely, up to translations and
reflections. The answer is positive for convex polytopes when
$m \leqslant 3$; see \cite{Bianchi}. Variations of this problem exist
and apply to our setting, such as determining when two model sets
share the same covariogram and the corresponding inverse problem
\cite{BG-Homometric}.

Before turning our attention to the realm of model sets, we briefly
recall and collect various properties of a covariogram that
immediately follow from its definition, some usual inequalities, and
elementary change of variables arguments; see \cite{TAO,Klick} for
further references.

\begin{prop}\label{prop:prop_cvg}
  Let\/ $W, W' \subset \RR^{m}$ be non-empty compact sets. Then, 
\begin{enumerate}
\itemsep=2pt
\item[(i)] $\cvg^{}_{W} \in C^{}_{\mathtt{c}}(\RR^m)$, and\/ $\cvg^{}_{W}$
   is an even function, 
\item[(ii)]
          $\mathrm{supp} ( \cvg^{}_{W}) \subseteq W \! - W \defeq
          \{x-y : x,y\in W \}$,
\item[(iii)] $W \nts\subseteq W'$ implies\/
          $\cvg^{}_{W}(x) \, \leqslant \, \cvg^{}_{W'}(x)$ for all\/
          $\in\RR^m$,
\item[(iv)] $\cvg^{}_{W}$ is Fourier transformable, and one has,
          for all\/ $x\in\RR^m$,
\[
  \widehat{\cvg^{}_{W}}(x) \, = \, \bigl|
  \widehat{\mathbf{1}^{}_{W}}(x) \bigr|^2 \ts,
\]
\item[(v)] for all\/ $b\in \RR^m$ and all\/ $\lambda \in \RR$ with\/
  $\lambda > 0$, one has
\[ 
  \cvg^{}_{W +\ts b}(x) \, = \, \cvg^{}_{W}(x)\ts, \quad \mbox{and}
  \quad \, \cvg^{}_{\lambda W} (x) \, = \, \lambda \ts\ts \cvg^{}_{W}
  \Big( \myfrac{x}{\lambda} \Big) ,
\]
\item[(vi)] if\/
  $\mathrm{int}(W) \cap \mathrm{int}(\mathrm{W'}) = \varnothing$
  with\/ $\vol(\partial W) = \vol(\partial W')=0$, one has
\[ 
   \cvg^{}_{W\ts \cup \ts W'}(x) \, = \, \cvg^{}_{W}(x) + \cvg^{}_{W'}(x) +  
   \bigl(\mathbf{1}^{}_{-W}\ast \mathbf{1}^{}_{W'}\bigr) (x) + 
   \bigl(\mathbf{1}^{}_{-W}\ast  \mathbf{1}^{}_{W'}\bigr) (-x)
\]
for almost every\/ $x \in \RR^m \nts$. \qed
\end{enumerate}
\end{prop}

Clearly, the complexity of determining the covariogram of a given set
$W$ depends on its structure. We are interested in finding the
covariograms of so-called \emph{Rauzy fractals}. They are solutions of
an iterated function system and are intimately related to Pisot
substitutions and their description as model sets. We briefly
summarise this connection in what follows, and refer the reader to
\cite{TAO} for further details and to \cite{Meyer,STTopo,SingThesis} for
deeper connections.

Let $\cA$ be a finite set of $n$ symbols, called an \emph{alphabet}.
A \emph{substitution rule} $\varrho$ is a map that sends each element
of $\cA$ (\emph{a letter}) to a non-empty concatenation of finitely
many letters (\emph{a word}), and extends to all (bi-infinite) words
by the endomorphism property. To every substitution rule $\varrho$,
one can assign the \emph{substitution matrix}
$M^{}_{\varrho} \in \mathrm{Mat}(n,\, \ZZ)$ with entries
\[ 
  \bigl( M^{}_{\varrho}\bigr)^{}_{ij} \ts = \,
  \text{number of letters of type $i$
       in } \varrho(j)\ts,
\]
where we have fixed a numbering of the letters of $\cA$. A
substitution $\varrho$ is \emph{primitive} if its substitution matrix
$M^{}_{\varrho}$ is primitive.  By iterating a primitive substitution
rule, starting with a~letter from the alphabet (or, more generally,
with a \emph{legal word} \cite[Def.~4.5]{TAO}), we obtain for
(a~suitable power of) the substitution a (bi)-infinite fixed point ---
a sequence of letters from the alphabet which is invariant under the
action of (the power of) the substitution rule.

Since the substitution matrix is a non-negative matrix, via standard
\textit{Perron--Frobenius} (PF) arguments \cite[Ch.~2.4]{TAO}, it
provides insight into various properties of $\varrho$.  In particular,
if a~primitive substitution is additionally a \textit{Pisot}
substitution, meaning the PF eigenvalue of $M^{}_{\varrho}$ is a
\emph{Pisot–Vijayaraghavan} (PV) number and the characteristic
polynomial is irreducible, the substitution leads to a tiling of the
real line by assigning a~closed interval to each letter whose (natural
tile) length is proportional to an entry of the left PF eigenvector of
$M^{}_{\varrho}\ts$. These intervals are called \textit{prototiles}.
Then, the PF eigenvalue of $M^{}_{\varrho}$ gives the inflation factor
$\lambda$ of the induced (geometric) inflation rule. In this setting,
it is natural to work with the \emph{displacement
  matrix}~$T^{}_{\varrho}$, which is a set-valued analogue of the
substitution matrix with entries
\[ 
  \bigl(T^{}_{\varrho}\bigr)^{}_{ij} \, = \, T^{}_{ij} \, = \,
  \left\{\mbox{relative positions of tiles of type $i$ in supertile }
    \varrho(j) \right\}\ts,
\]
and one has $(M^{}_{\varrho})^{}_{ij} = \card(T^{}_{ij})$.  Placing a
\emph{control point} on the left endpoint of each tile, we can
understand the fixed point of a given substitution as a collection of
point sets assigned to each letter. If we collect all control points
of tiles of type $i$, we obtain a Delone set $\vL^{}_{i}\subset \RR$.
The set of all control points will be denoted by $\vL$. It satisfies
\[ 
     \vL \, =  \bigcup_{k=1}^n \vL^{}_{k} 
\]
and the union is disjoint. Since one considers a tiling of the real
line described as a fixed point of a Pisot substitution (inflation)
rule $\varrho$ with $n$ letters (prototiles) with the inflation factor
$\lambda$, this is reflected on the level of the control points as
\begin{equation}\label{eq:EMS}
    \vL^{}_i \, = \, \bigcup_{j=1}^{n} \, \lambda \vL+ T^{}_{ij} \, = \, 
    \bigcup_{j=1}^{n} \, \bigcup_{t \in T^{}_{ij}} \lambda \vL^{}_j + t \ts, 
\end{equation}
where, again, the unions are disjoint. Note that the RHS, viewed as a
mapping between $n$-tuples of Delone sets, does not define a
contraction with respect to any natural metric. Consequently, a
characterisation of the solution space of \eqref{eq:EMS} is
difficult. We are not aware of any general result in this direction.

If the alphabet is binary and $\lambda$ is a PV number, the sets
$\vL^{}_{i}$ are \emph{regular model sets} \cite{HS2003}. For larger
alphabets, this problem is known under the name \emph{Pisot
  substitution conjecture}. It is unclear whether or not it should
hold, but all known examples with irreducible characteristic
polynomial of the substitution matrix satisfy it. We refer the reader
to~\cite{APisot} for a summary on the Pisot substitution conjecture,
to \cite{Barge} for various families for which the conjecture has been
verified, and to \cite{BalRus} for an extensive computer search for a
counterexample.

Model sets are special cases of cut-and-project sets with nice
properties, both arising from cut-and-project schemes; compare
\cite{Meyer,Moo97}.

\begin{definition}\label{def:CPS} 
  A \emph{cut-and-project scheme} (CPS) is a triple
  $(\RR^d,\, H,\, \cL)$ consisting of $\RR^d$ (\textit{direct/physical
    space}) and a locally compact Abelian group (LCAG) $H$
  (\textit{internal space}), together with a discrete co-compact
  subgroup --- a lattice --- $\cL \subset \RR^d \times H$ and the
  natural projections $\pi: \RR^{d} \times H\to \RR^{d}$ and
  $\pi^{}_{\text{int}} :\RR^{d} \times H\to H$ enjoying the following
  properties:
\begin{enumerate}
\itemsep=2pt
\item The restriction $\pi|_{\cL}$ is one-to-one. 
\item The image $\pi_{\text{int}}(\cL)$ is dense in $H$.
\end{enumerate}
\end{definition}

Given a CPS $(\RR^d,\, H,\, \cL)$, we set $L= \pi( \cL)$, which is a
subgroup of $\RR^d$. The first condition in the definition of a CPS
implies the existence of a mapping $\star: L \to H$, called the
\emph{$\star$-map} (known as the star map), defined by
$\star= \pi_{\text{int}} \circ ( \pi|_{\cL})^{-1}$. If $H = \RR^m$,
the CPS is called \textit{Euclidean}. We summarise a CPS via the
diagram
\begin{equation*}
\renewcommand{\arraystretch}{1.2}
\begin{array}{ccccc@{}l}
  \RR^{d} & \xleftarrow{\;\;\; \pi \, \ts \;\;\; }
  & \RR^{d} \nts\nts \times \nts\nts H 
& \xrightarrow{\;\: \pi^{}_{\text{int}} \;\: } & H & \\
  \cup & & \cup & & \cup  & \hspace*{-2ex}
           \raisebox{1pt}{\text{\scriptsize dense}} \\
  \pi (\cL) & \xleftarrow{\;\ts 1-1 \;\ts } & \cL & \xrightarrow{ \qquad }
  &\ts \pi^{}_{\text{int}} (\cL) & \\ \| & & & & \| & \\ 
  L  & \multicolumn{3}{c}{\xrightarrow{\qquad\quad\quad \
       \star \quad\quad\qquad}} 
&  {L}^{\star\nts}   &  
\end{array}
\renewcommand{\arraystretch}{1}
\end{equation*}
as usual. With the $\star$-map, we can now define a model set. We
refer the reader to \cite[Sec.~7]{TAO} and~\cite{Moo97} for further
discussions of model sets and their properties.

\begin{definition}\label{def:model_set}
  Given a CPS $(\RR^d,\, H,\, \cL)$ and an arbitrary, relatively
  compact set $W \subset H$, we denote by $\oplam(W)$ the set
\[ 
    \oplam(W) \, \defeq \, \{ x \in L : x^\star \in W \}\ts, 
 \]
 which we call a \textit{cut-and-project set}. If $W$ has non-empty
 interior, $\oplam(W)$, or any translate $t+\oplam(W)$ with
 $t \in \RR^{d}$, is called a \textit{model set}. When the boundary
 $\partial W$ has zero Haar measure in $H$, the model set is called
 \textit{regular}.
\end{definition}

Now, one version of the Pisot substitution conjecture states that, if
we have a set of control points $\vL^{}_{i}$ of a~tiling arising from
a Pisot substitution, we can find a suitable LCAG $H$ and relatively
compact sets $W^{}_{i} \subset H$ with non-empty interior such that
$\vL^{}_{i} = \oplam(W^{}_{i})$ up to a~set of zero density. We
restrict ourselves to unimodular substitutions, which means we have
$\lvert \det(M_{\varrho})\rvert = 1$.  Under this additional
assumption, we obtain a Euclidean CPS \cite{TAO, SingThesis}.  This is
also plausible from the results in \cite{BK}, which establish an
appropriate connection to a~minimal rotation on a torus, and thus a
fully Euclidean embedding.

Since one can, without loss of generality, assume that all $\vL^{}_{i}$
belong to the ring of integers $\mathcal{O}^{}_{\QQ(\lambda)}$ of the
algebraic number field $\QQ(\lambda)$, where $\lambda$ is the
inflation factor, we employ a~\emph{Minkowski embedding} of
$\mathcal{O}^{}_{\QQ(\lambda)}$ to obtain the lattice $\cL$. The
projection $\pi$ on its first coordinate gives the control points,
whereas the remaining ones give the projection $\pi^{}_{\mathrm{int}}$
into the internal space, which is $\RR^{m}$ (if $\lambda$ is of degree
$n$, then $m=n-1$). Now, for the windows, one considers the
$\star$-image of Eq.~\eqref{eq:EMS} and its closure. Then, with
$W^{}_{i} \defeq \overline{\vL^{\star}_i}$, we obtain
\begin{equation}\label{eq:IFS_window}
    W^{}_i \, =  \bigcup_{j=1}^n \bigcup_{t \in T_{ij}}
    \lambda^{\star} W^{}_j + t^{\star} ,
\end{equation}
which is an iterated function system on the space
$\big(\mathcal{K}(\RR^{m})\big)^n$ of all $n$-tuples of non-empty
compact subsets of $\RR^{m}$ equipped with Hutchinson's metric; see
\cite{BM-Self} for details. Since the RHS defines a contraction, by
Hutchinson's theorem \cite{Hutch} (a version of the Banach's
contraction principle), one has a unique solution. The resulting sets
are called \emph{Rauzy fractals} and have been studied extensively;
see for example \cite{AHRauzy,BG2020, STTopo}. Some of their
properties include the fact that they are topologically regular (i.e.,
no isolated points) and compact sets of positive measure with a
(usually) fractal boundary \cite{BST}. If the internal space is
one-dimensional, one can even speak of the Rauzy fractal being a
\textit{Cantorval}, for which we refer the reader to
\cite{BGM2024}. Let us summarise some useful properties of Rauzy
fractals as follows.

\begin{prop}[{\textnormal{\cite[Cor.~6.66]{SingThesis}}}]\label{prop:IFS_window}
  Let\/ $\{W^{}_{i} : 1 \leqslant i \leqslant n \}$ be the unique
  solution of \eqref{eq:IFS_window}, as obtained from a unimodular
  Pisot substitution. Then, the following properties hold.
\begin{itemize}
\itemsep=2pt
\item[(i)] All\/ $W^{}_{i}$ have positive Lebesgue measure.
\item[(ii)] In the IFS \eqref{eq:IFS_window}, the unions on the
     right-hand side are measure disjoint. 
\item[(iii)] The boundaries\/ $\partial W^{}_{i}$ have zero
     Lebesgue measure.
\item[(iv)] The sets\/ $W^{}_{i}$ are topologically regular, i.e.,
     they are the closures of their interiors and thus contain no
     isolated points.
\item[(v)] The interiors of the\/ $W^{}_{i}$ are disjoint.  \qed
\end{itemize}
\end{prop}

In this context, we also refer the reader to further pioneering works
on the geometry of unimodular Pisot substitutions, such as
\cite{AI01,IR06,SW02}.

Since the topology of Rauzy fractals can be extremely complicated,
such as having infinitely many connected components with non-integer
boundary dimensions \cite{STTopo}, their covariograms are, by standard
means, impossible to obtain. On the other hand, since Rauzy fractals
correspond (in some sense) to fixed points of substitutions, we have
an additional tool at hand --- the \emph{pair correlation
  functions}. For a~model set $\vL = \bigcup_{j=1}^n \vL^{}_{j}$, let
\[  
  \nu^{}_{ij}(y) \, \defeq  \lim_{r\to \infty} \frac{1}{\card (\vL^{(r)})} 
    \sum_{\substack{x  \in \vL^{(r)}_{i}\\ x+y \in \vL^{}_j}} \! \! 1
\]
denote the relative frequency of a $2$-point patch consisting of a
control point of type $i$ on the left and a control point of type $j$
on the right with distance $y \in \vL^{}_j - \vL^{}_{i}$, where a
negative~$y$ flips the role of left and right. This quantity exists
uniformly by the uniform distribution property of model sets
\cite[Thm.~7.3]{TAO}.  By applying the $\star$-map, with
$\vL^{(r)}_j = \vL^{}_{j} \cap B^{}_{r}(0)$, we obtain the
representation
\begin{align*} 
    \nu^{}_{ij}(y) \,& = \lim_{r\to \infty} \frac{1}{\card (\vL^{(r)})} 
    \sum_{\substack{x^{\star}  \in (\vL^{(r)}_{i})^{\star}\\ 
    (x+y)^{\star} \in W^{}_{j}}}
    \! \! 1 \; = \; \frac{1}{\text{vol}(W)} \int_{\RR^{n-1}}
    \mathbf{1}_{W^{}_{i}}(z) \,
    \mathbf{1}_{W^{}_{j}}(z+y^{\star})\, \dd z \\ &  = \,
     \frac{\big(\mathbf{1}^{}_{-W^{}_{i}} * 
     \mathbf{1}^{}_{W^{}_{j}}\big)(y^{\star})}{\text{vol}(W)} \,
     \eqdef \, g^{}_{ij}(y^{\star})\ts ,
\end{align*}
where the second equality in the first line holds by Weyl's theorem on
uniform distribution \cite{KN,TAO}.  Up to a factor of $\vol(W)^{-1}$,
the functions $g^{}_{ij}$ are precisely the covariograms of
Definition~\ref{def:cvg}, so
\begin{equation}\label{eq:pair_corr_def_conv}
   \nu^{}_{ij}(x) \, = \, \begin{cases} \frac{1}{\vol(W)} \, 
   \bigl(\bs{1}^{}_{-W^{}_{i}} \ast 
   \bs{1}^{}_{W^{}_{j}}\bigr)(x^{\star})\ts,  & \mbox{if} \ x \in 
   \vL^{}_{j}-\vL^{}_{i}\ts, \\ 0 \ts, & \mbox{otherwise}. \end{cases}
\end{equation}

\begin{remark}\label{rem:rem_sample}
  Recall that, by Proposition~\ref{prop:prop_cvg}, the covariogram is
  a \textit{continuous} function. Moreover, if
  $x \in \vL^{}_{j}-\vL^{}_{i}$, the points $x^{\star}$ are dense in
  $W_{j} - W_{i}$, by definition of the CPS, \textit{and} uniformly
  distributed by Weyl's theorem \cite{KN,TAO}. Therefore, plotting the
  covariogram on a~sufficiently large sample of these $x^{\star}$, one
  expects a good illustration of the continuous function and its
  behaviour, though care is definitely required in view of the vast
  range of phenomena that are possible in the realm of continuous
  functions. \exend
\end{remark}

By realising the covariogram of the window as the autocorrelation
function of the model set, we can circumvent issues with the window in
internal space, and instead work with the renormalisation relations
for $\nu^{}_{ij}$. These techniques were introduced in \cite{BG16} and
further developed and used, for example, in
\cite{BGMan,ManiboThesis,Maz25}.  We recall the main result that
profits from the inflation structure and enables computing these
quantities.

\begin{prop}[{\cite[Prop.~2.2.1]{ManiboThesis}} and
  {\cite[Lemma~3.16]{BGMan}}]\label{prop:renorm} Let\/
  $\{\vL^{}_{i}\}_{i=1}^n$ be a fixed point of a~primitive geometric
  inflation\/ $\varrho$ with inflation factor\/ $\lambda >1$ arising
  from a substitution over an\/ $n$-letter alphabet. Then, the pair
  correlation functions\/ $\nu^{}_{ij}$ exist, and satisfy the
  \emph{exact} renormalisation equations
\[
     \nu^{}_{ij}(x) \, = \, \myfrac{1}{\lambda} \sum_{k, \ell =1}^n \, 
     \sum_{r\in T^{}_{ik}} \, \sum_{s\in T^{}_{j\ell}} \nu^{}_{k\ell} 
     \left(\myfrac{x+r-s}{\lambda} \right) , 
\]
where\/ $T^{}_{\varrho} = (T^{}_{ij})$ is the displacement matrix of
the inflation\/ $\varrho$. \qed
\end{prop}

As $\lambda>1$, the arguments of the $\nu^{}_{ij}$ on the RHS of the
equations in Proposition~\ref{prop:renorm} are generally smaller in
modulus than those on the LHS. Since $\nu^{}_{ij}(x) =0$ for all
$x\notin \vL^{}_{j} - \vL^{}_{i}$, the system of equations splits
naturally into a self-consistent part, which closes on itself, and the
recursive part. As the correlation functions $\nu^{}_{ij}$ satisfy the
renormalisation relations, their counterparts $g^{}_{ij}$ satisfy the
same equations on a subset of $\RR^{m}$ given by the $\star$-image of
$\vL^{}_{j} - \vL^{}_{i}$. These equations can also be derived without
prior knowledge of the renormalisation equations from
Proposition~\ref{prop:renorm}. It suffices to know the IFS for the
windows, Eq.~\eqref{eq:IFS_window}, which gives some kind of a dual to
the equations from Proposition~\ref{prop:renorm}.  Indeed, the IFS can
be turned into a set of equations for the characteristic functions of
the windows. These, in turn, induce equations for their (cross)
covarigrams, which we shall state next.  We recall that
$\lambda^{\star}$ refers to the linear mapping
$y \mapsto \lambda^{\star}y$ in $\RR^m$, which satisfies
$(\lambda x)^{\star} = \lambda^{\star} x^{\star}$ for all
$x\in \QQ(\lambda)$.

\begin{prop}\label{prop:internal}
  Let\/ $\vL^{}_{i} = \oplam(W^{}_{i})$ be a collection of regular
  model sets arising from a single CPS\/ $(\RR, \, \RR^{m}, \, \cL)$
  describing the control points of a fixed point of a unimodular
  inflation with a~PV inflation factor\/ $\lambda$, and with
  displacement matrix\/ $T^{}_{\varrho} = (T^{}_{ij})$. Then, the
  functions\/
  $g^{}_{\alpha \beta} \defeq \bs{1}^{}_{-W^{}_{\alpha}} \ast
  \bs{1}^{}_{W^{}_{\beta}}$ satisfy, for almost all \/$y \in \RR^{m}$,
  the relations
\[
  g^{}_{ij}(y) \, = \, \myfrac{1}{\lambda} \sum_{k,\ts \ell=1}^{n} \,
  \sum_{r\in T^{}_{ik}} \, \sum_{s \in T^{}_{j\ell}} \, g^{}_{k \ell}
  \bigl( ({\lambda^{\star}})^{-1}(y+ r^{\star} {-} \ts s^{\star})
  \bigr)\ts.
\]
\end{prop}

\begin{proof}
  We recall the properties of the covariogram function from
  Proposition~\ref{prop:prop_cvg} and observe that
  $\bigl(\bs{1}^{}_{-\lambda^{\star}\nts A} \, \ast \,
  \bs{1}^{}_{\lambda^{\star}B} \bigr)(y) \, = |\det (\lambda^{\star})
  | \bigl( \bs{1}^{}_{-A} \ast \bs{1}^{}_{B} \bigr)
  \big(({\lambda^{\star}})^{-1}y\bigr)$ holds for arbitrary compact
  sets $A,B$, which follows by a standard change of variables
  argument. We also recall that the unions on the RHS of
  \eqref{eq:IFS_window} are measure disjoint and their boundary is of
  Lebesgue measure zero (Proposition~\ref{prop:IFS_window}). 
  Then, we obtain
\begin{align*}
     g^{}_{ij}(y)\, & = \,  \bigl( \bs{1}^{}_{-W^{}_{i}} \, \ast \,   
     \bs{1}^{}_{W^{}_j} \bigr) (y) \, = \, \Bigl(\bs{1}^{}_{-\bigcup_{k=1}^n 
     \bigcup_{r \in T_{ik}} \lambda^{\star} W^{}_k + r^{\star}} \, \ast \,   
     \bs{1}^{}_{\bigcup_{\ell=1}^n \bigcup_{s \in T_{j\ell}} 
     \lambda^{\star} W^{}_{\ell} + s^{\star}} \Bigr) (y)\\
    & = \sum_{k,\ts \ell=1}^{n} \, \sum_{r\in T^{}_{ik}} \, \sum_{s \in T^{}_{j\ell}} 
      \, \lvert \det (\lambda^{\star}) \rvert \,
      \bigl(\bs{1}^{}_{-W^{}_{k}}  \ast \,   
      \bs{1}^{}_{W^{}_{\ell}}\bigr) \bigl( ({\lambda^{\star}} )^{-1}
      (y+ r^{\star} {-} \ts s^{\star}) \bigr) \\
   & = \, \myfrac{1}{\lambda} \sum_{k,\, \ell=1}^{n} \, \sum_{r\in T^{}_{ik}} \, 
   \sum_{s \in T^{}_{j\ell}} \, g^{}_{k\ell}\bigl( ({\lambda^{\star}})^{-1}
   (y+ r^{\star} {-} \ts s^{\star}) \bigr)\ts,
\end{align*}
which holds almost everywhere. For the last equality, let us recall
that the Pisot property and unimodularity together imply that
$\lambda^{-1} = \lvert \det (\lambda^{\star}) \rvert$.
\end{proof}

We refer the reader to \cite{Jan} for a brief summary of the
non-unimodular setting with LCAG~$H$ having a $p$\ts-adic
component. Let us turn to concrete examples and explain further
details.

\section{Binary Alphabet Examples}\label{section3}

In this section, we provide some examples of determining the
covariograms of Rauzy fractal windows from Pisot substitutions over a
binary alphabet. We first construct a model set with a highly
irregular window, and implement the renormalisation procedure to
calculate the covariogram. Note that in this one-dimensional internal
space case, as conjectured in \cite{BGM2024,Jan}, there are two
possibilities for the covariogram: a piecewise linear function,
arising from interval-type windows, or a~`spiky' function, as
illustrated in Figure~\ref{fig:SilverCovario} below. The latter arises
when the window is a Cantorval. For an example where the internal
space is two-dimensional, we refer the reader to \cite{Klick}.

\subsection{The sister silver mean tiling}
Let us construct the \emph{sister silver mean} (SSM) chain. Consider
the binary alphabet $\cA = \{a\ts,b\}$, and the substitution rule
\[
     \varrho^{}_{s} \,:\, \begin{cases}
             a  \mapsto  bba  \ts,\\
             \ts b \ts \mapsto  ab \ts,  \end{cases}
\]
which we abbreviate as
$\varrho^{}_{s} = (bba,ab) = \big(\varrho^{}_{s}(a),
\varrho^{}_{s}(b)\big)$.  Starting with the legal seed
$\omega^{}_{(0)} = a|a$, where $|$ denotes the origin, by iteration
one obtains a sequence of growing words $\omega^{}_{(n)}$
\[ 
  a\ts|\ts a \xmapsto {\, \varrho^{}_{s}\, } bba\ts | \ts bba
  \xmapsto{\, \varrho^{}_{s}\, } ababbba\ts|\ts ababbba\xmapsto{\,
    \varrho^{}_{s}\, } {\dots} \xmapsto{\, \varrho^{}_{s}\, }
  \omega_{(n)} \xrightarrow{\, n \to \infty \,} \omega \xmapsto{\,
    \varrho^{}_{s}\, } \omega^{\prime} \xmapsto{\, \varrho^{}_{s}\, }
  \omega \ts,
\]
which converges, with obvious meaning in the product topology, towards
a $2$-cycle of bi-infinite words
$\omega, \omega^{\prime} \in \cA^{\ZZ}$.  The left-infinite half is
fixed, while the right-infinite half alternates between two words that
differ on a set of zero density only. In fact, they form a proximal
pair; see \cite[Sec.~4.3]{TAO} for details.

Now, we can move to the geometric description and easily deal with
this difference in the cut-and-project description.  The substitution
matrix of $\varrho^{}_{s}$ is
$ M^{}_{s} =\big(\begin{smallmatrix}1 & 1\\ 2 &
  1\end{smallmatrix}\big)$, with PF eigenvalue $\lambda =
1+\sqrt{2}$. The left and (statistically normalised) right PF
eigenvectors are
\begin{equation} \label{eq:PF-eigenvec} v^{}_{\ell} =
  \bigl(\sqrt{2},\, 1\bigr) \quad \text{ and } \quad v^{}_{r} \, = \,
  \bigl(\lambda^{-1}, \,\lambda^{-1}\ts\sqrt{2}\, \bigr)^{\top} \ts
  =\,\bigl(\lambda-2, \,3-\lambda\, \bigr)^{\top} .
\end{equation}

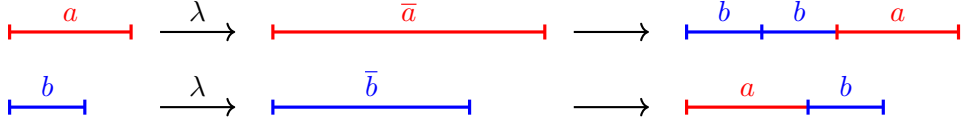
\begin{figure}
\begin{tikzpicture}[scale=0.5]
  \draw[blue, very thick] (-3,0) -- (-1,0)node[above, pos=0.5]{$b$}
  ; %bottom left b %
  \draw[blue, very thick] (-3,.2) -- (-3,-.2); \draw[blue, very thick]
  (-1,.2) -- (-1,-.2); %bottom left b end %
  \draw[thick,->] (1,0) -- (3,0)node[above,
  pos=0.5]{$\lambda$}; %bottom first arrow
  \draw[thick,->] (1,2) -- (3,2)node[above,
  pos=0.5]{$\lambda$}; %top first arrow
  \draw[red, very thick] (-3,2) -- (0.23,2)node[above,
  pos=0.5]{$a$}; %top left a tile&
  \draw[red, very thick] (-3,2.2) -- (-3,1.8); \draw[red, very thick]
  (0.23,2.2) -- (0.23,1.8); %top left a end
  \draw[blue, very thick] (4,0) -- (9.23,0)node[above,
  pos=0.5]{$\overline{b}$}; %bottom middle
  \draw[blue, very thick] (4,.2) -- (4,-.2); \draw[blue, very thick]
  (9.23,.2) -- (9.23,-.2); %end bottom middle
  % \draw[blue, very thick] (4,0) -- (7.23,0)node[above,
  % pos=0.5]{$\overline{b}$}; bottom middle \draw[blue, very thick]
  % (4,.2) -- (4,-.2); \draw[blue, very thick] (7.23,.2) --
  % (7.23,-.2); %end bottom middle
  \draw[thick,->] (12,2) -- (14,2); %top second arrow
  \draw[thick,->] (12,0) -- (14,0); %bottom second arrow
  \draw[blue, very thick] (15,2) -- (17,2) node[above,
  pos=0.5]{$b$}; %top right b
  \draw[blue, very thick] (15,1.8) -- (15,2.2); \draw[blue, very
  thick] (17,1.8 ) -- (17,2.2); %end top right b
  \draw[blue, very thick] (17,2) -- (19,2) node[above,
  pos=0.5]{$b$}; %top right b middle
  \draw[red, very thick] (19,1.8 ) --
  (19,2.2); %end top right b middle
  \draw[red, very thick] (19,2) -- (22.23,2) node[above,
  pos=0.5]{$a$}; %top right a
  \draw[red, very thick] (22.23,1.8 ) -- (22.23,2.2); %end top right a
  \draw[red, very thick] (15,0) -- (18.23,0) node[above,
  pos=0.5]{$a$}; %lower right a
  \draw[red, very thick] (15,-.2) -- (15,.2); \draw[blue, very thick]
  (18.23,-.2) -- (18.23,.2); %end lower right a
  \draw[blue, very thick] (18.23,0) -- (20.23,0) node[above,
  pos=0.5]{$b$}; %lower right b
  \draw[blue, very thick] (20.23,-.2) --
  (20.23,.2); %end lower right b
  \draw[red, very thick] (4,2) -- (11.23,2)node[above,
  pos=0.5]{$\overline{a}$} ; %top middle
  \draw[red, very thick] (4,2.2) -- (4,1.8); \draw[red, very thick]
  (11.23,2.2) -- (11.23,1.8); %end top middle
\end{tikzpicture}
\caption{Geometric inflation rule for the self-similar SSM
  tiling.  \label{fig tiling pic}}
\end{figure}

The $a$ and $b$ prototiles for this choice have length $\sqrt{2}$ and
$1$. The inflated $a\ts$-tile is called the (level-one) $a$-supertile,
and, analogously, we have the (level-one) $b\ts$-supertile. Below,
they will be called the $a$ and $b$ supertiles, and denoted by
$\overline{a}$ and $\overline{b}$ respectively; the inflation rule is
depicted in Figure~\ref{fig tiling pic}. Denote the set of left
endpoints (control points) of tilings corresponding to $\omega$ and
$\omega^{\prime}_{}$ by $\vL^{}_{\omega}$ and
$\vL^{}_{\omega^{\prime}_{}}$. These sets are in a natural one-to-one
correspondence with $\omega$ and $\omega^{\prime}_{}$, and each has a
density of $\tfrac{2 + \sqrt{2} \ts }{4}$. Thus, the discrepancies
between the two fixed points form a zero density set; compare
\cite[Ch.~4]{TAO}.

Denote the sets of control points of the tiles of types $a$ and $b$ by
$ \vL^{}_{a}$ and $ \vL^{}_{b}$, with
$ \vL = \vL_{a} \,\dot{\cup}\, \vL_{b}$.  By our choice of the control
point, it immediately follows that $\vL \subset \ZZ[\sqrt{2} \,]$,
which is the smallest $\ZZ\ts$-module into which we can embed
$\vL$. The standard Minkowski embedding of
$\vL \subset \ZZ[\sqrt{2} \,]$ gives the lattice
$\cL = \ZZ(1,1) \oplus \ZZ(\sqrt{2}, -\sqrt{2})$; see
\cite[Ch.~7]{TAO} and \cite[Thm.~2.13.1]{Koch} for details.  In
particular, the $\star$-map is given by the non-trivial Galois
automorphism of $\QQ(\sqrt{2}\ts)$, which is
$a+b\sqrt{2} \mapsto (a+b\sqrt{2}\ts)^{\star} = a-b\sqrt{2}$.  This
gives the Euclidean CPS
\begin{equation} 
\label{eq CPS}
\renewcommand{\arraystretch}{1.2}
   \begin{array}{ccccc@{}l}
     \RR & \xleftarrow{\;\;\; \pi \;\;\; }
     & \RR \nts\nts \times \nts\nts \RR &
   \xrightarrow{\;\: \pi^{}_{\text{int}} \;\: } & \RR & \\
     \cup & & \cup & & \cup  & \hspace*{-2ex}
        \raisebox{1pt}{\text{\scriptsize dense}} \\
     \pi (\cL) & \xleftarrow{\;\ts 1-1 \;\ts \ts} & \cL
       & \xrightarrow{ \qquad } 
       &\pi^{}_{\text{int}} (\cL) & \\ \| & & & & \| & \\ 
     L = \ZZ[\sqrt{2}\,] &
      \multicolumn{3}{c}{\xrightarrow{\qquad\quad\quad \  \star 
   \quad\quad\qquad}} &  {L}^{\star \nts} = \ZZ[\sqrt{2}\,]  &  \end{array}
\renewcommand{\arraystretch}{1}
\end{equation} 
By construction, $\pi^{}_{\text{int}}(\cL)$ is dense in internal
space. From a CPS of the form~\eqref{eq CPS}, we can construct a
regular model set, with suitable $W \subset \RR$, as
\[
  \oplam(W)\, \defeq \, \{x \in L : x^{\star} \in W\} \, \subset \,
  \ZZ[\sqrt{2}\,]\ts.
\]

\begin{figure}
\includegraphics[width = 0.85 \textwidth]{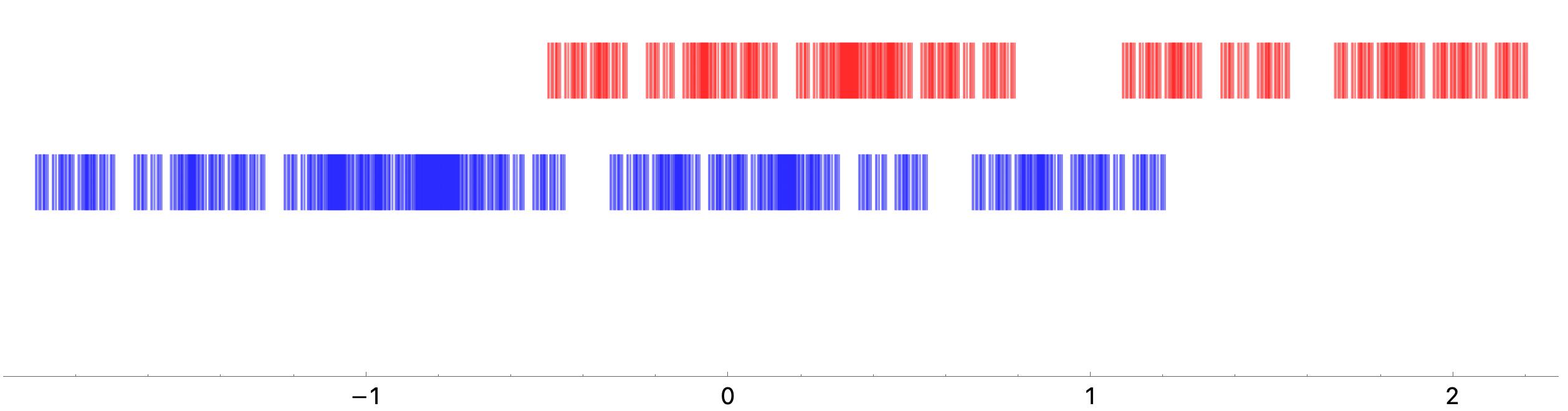}
\caption{The windows of the SSM inflation tiling, which are the fixed
  point of the IFS \eqref{eq:IFS_SSM}; $W^{}_{a}$ is red (top) and
  $W^{}_{b}$ is blue (bottom).  The windows are one-dimensional, but
  we assign some fixed arbitrary height to the points for
  illustration. The windows are measure-theoretically disjoint, but
  the resolution is limited by the large Hausdorff dimension of the
  window boundaries. \label{fig:SilverWindow} }
\end{figure}

As described in Section~\ref{section2}, the window corresponding to
the fixed point of the inflation rule, up to the zero-density set of
discrepancies, is found by first considering the following set-valued
iterations,
\[
\begin{split}
  \vL^{}_{a} \, & =\, \big( \lambda  \vL^{}_{a} + 2 \big) \; \dot{\cup}\;
  \lambda   \vL^{}_{b} \ts,  \\
  \vL^{}_{b} \, & =\, \lambda   \vL^{}_{a} \; \dot{\cup} \; \big( \lambda
   \vL^{}_{a} + 1 \big) \; \dot{\cup} \; \big( \lambda
   \vL^{}_{b} + \lambda - 1\, \big)\ts.
\end{split}
\]
The right hand sides describe the positions of the tiles in the
inflated version of the tilings.  For example, applying the inflation
rule to an $a$ tile produces another $a$ preceded by two $b$'s, giving
the first term in the first equation. Similarly, inflating a $b$ tile
produces an $a$ at the inflated point, giving the last term in the
first equation. Taking the $\star$-image and closure of the sets,
denoted $W^{}_{a,b}:= \overline{ \vL^{\star}_{a,b}}$, gives the
following internal space IFS,
\begin{equation}\label{eq:IFS_SSM}
\begin{split}
    W^{}_{a} \, & =\, \big( (2-\lambda) W^{}_{a} + 2 \big)\, \cup\, 
         (2-\lambda) W^{}_{b} \ts  , \\
    W^{}_{b} \, & = \, (2-\lambda) W^{}_{a} \, \cup\, \big( (2-\lambda)
        W^{}_{a} + 1 \big) \, \cup\,  \big( (2-\lambda)  W^{}_{b} +
        1 - \lambda\, \big) .
\end{split}
\end{equation}
The fixed point of this contractive IFS, as $\lambda^{\star}<1$, is a
Rauzy fractal; see Figure~\ref{fig:SilverWindow}. By
Proposition~\ref{prop:IFS_window}, these windows define regular model
sets, and so does their union.

There are several ways to determine the Hausdorff dimension of the
boundary of the window. One method, laborious but allowing a complete
reconstruction of the window, involves the graph iterated function
system of the window; see \cite{BGM2024, SingThesis}. Here, we use the
\textit{overlap algorithm} from \cite{AOverlap, BGG-Orbit,
  SolSelfSimilar}, based on the discrepancy matrix; this is applicable
to more complex tilings, but we need only consider our one-dimensional
case.  Consider two overlaid copies of the SSM tiling, one being a
copy shifted by an element of the return module. We say a pair of
tiles, one from each copy, form an \textit{overlap} if the
intersection of their supports has a non-empty interior. We can then
consider the overlap itself as a tile determined by the support, the
type of tiles from which it comes, and the offset of the two tiles. If
the pair of tiles have the same type and position, we call it a
\textit{coincidence overlap}; otherwise, it is referred to as a
\textit{discrepancy overlap}. Therefore, from the original
substitution, we can define another substitution on the overlaps; the
corresponding substitution matrix is called the \textit{discrepancy
  matrix}.

The Hausdorff dimension of the  boundary 
$\partial W$ is then given \cite[Prop.~6.1]{BGG-Orbit} by 
\[
  \dim_{\mathrm{H}} (\partial W) \, = \, \overline{\dim}_{\mathrm{H}}
  (\partial W) \, =\, \frac{ \, \log (\gamma)}{\log
    (\lambda_{\mathrm{PF}})} \ts,
\]
where $\overline{\dim}_{\mathrm{H}} (\partial W)$ is the maximum of
$\overline{\dim}_{\mathrm{H}} (\partial W^{}_{a,b}\ts)$,
$\lambda_{\mathrm{PF}}$ is the PF eigenvalue of the substitution
matrix, and $\gamma$ is the spectral radius of the discrepancy
matrix. The first equality holds because $\vL^{}_{a}$ and $\vL^{}_{b}$
are MLD to $\vL$, see \cite{BGM2024} for details, thus the Hausdorff
dimension of $\partial W^{}_{a}$, $\partial W^{}_{b}$, and
$\partial W$ must be equal \cite[Rem.~7.6]{TAO}.  Implementing this
algorithm, an elementary calculation leads to the discrepancy matrix
\[
    M_{\text{dc}} \,= \, \begin{pmatrix}
          1& 0&0&0&1&1&0\\
          1& 0&0&0&0&0&0\\
          1& 0&0&0&0&0&0\\
          1& 0&0&0&0&0&0\\
          0& 1&0&0&0&0&0\\
          0& 0&1&0&0&1&1\\
          0& 0&1&0&1&1&0 \end{pmatrix}
\]
with characteristic polynomial
$\det (M_{\text{dc}} - x \one) = -x^{4}(x^{3}-2\ts x^{2}-1)$.  The
spectral radius of $M_{\text{dc}}$ thus has a value of approximately
$2.205{\ts}57$, giving
\[
  \dim_{\mathrm{H}} (\partial W) \,=\, \frac{ \, \log (\gamma)}{\log
    (\lambda_{\mathrm{PF}})} \, \approx \, 0.897{\ts}45 \ts.
\]
    
Obviously, computing the covariogram of this window according to
Definition~\ref{def:cvg} looks difficult if not impossible. To
overcome this problem, we now introduce the renormalisation
procedure. For the pair correlations $\nu_{\alpha \beta}$, where
$\alpha, \beta \in \{a,b\}$, we have $\nu_{\alpha \beta}(z) = 0$ for
any $z \notin \vL_{\beta}- \vL_{\alpha}$ and $\nu_{\alpha \beta}(0)=0$
for $\alpha \neq \beta$. Moreover, the right PF eigenvector of
\eqref{eq:PF-eigenvec} gives $\nu^{}_{aa}(0)=\lambda^{-1} = \lambda-2$
and $\nu^{}_{bb}(0)~=~\lambda^{-1}~\sqrt{2}~=~3-\lambda$. We also have
the symmetry relation
$\nu_{\alpha \beta}(-z)= \nu_{ \beta \alpha}(z)$, and the summatory
relationship
\[
  \nu(z) \,=\, \nu^{}_{aa}(z) \,+\, \nu^{}_{ab}(z)\,+\,
  \nu^{}_{ba}(z)\,+\, \nu^{}_{bb}(z)
\]
from Proposition~\ref{prop:prop_cvg}{\ts}(vi).

With the symmetry relations and inflation structure, we can determine
the renormalisation relations of $\nu^{}_{\alpha\beta}$ as a special
case of Proposition~\ref{prop:renorm}; see \cite{Klick} for worked
examples.

\begin{table}
\begin{center}
\renewcommand{\arraystretch}{1.7}
\begin{tabular}{ |c|c|c|c|c|c|c|c|c| } 
\hline
  $z$ & $0$  & $1$ &$2$ & $3$ &$\lambda-1$ &$\lambda$
      &$2\lambda-2$ &$\lambda+1$\\
\hline
$\nu^{}_{aa}(z)$ & $\lambda-2$ & $0$ & $0$ & $0$ & $\frac{3\, \lambda-7}{2}$ 
& $\frac{3 \ts \lambda-7}{2}$ & $0$ & $\frac{17-7\, \lambda}{2}$ \\ 
  $\nu^{}_{ab}(z)$ & $0$ & $0$ & $0$ & $0$ & $\frac{3-\lambda}{2}$ &
  $5-2\,\lambda$ 
& $\frac{3\, \lambda-7}{2}$ & $\frac{3\, \lambda-7}{2}$ \\ 
  $\nu^{}_{ba}(z)$ & $0$ & $\frac{3-\lambda}{2}$ & $5-2\,\lambda$
                      & $\frac{3\,\lambda-7}{2}$ 
& $0$ &$\frac{3\,\lambda-7}{2}$& $0$ & $5-2\,\lambda$ \\ 
  $\nu^{}_{bb}(z)$ & $3-\lambda$ & $\frac{3-\lambda}{2}$
                 & $\frac{3\,\lambda-7}{2}$ & $0$ 
& $0$ & $5-2\,\lambda$ & $0$ & $5-2\,\lambda$ \\ 
\hline
\end{tabular}
\vspace*{0.5cm}
\caption{The self-consistent part of the renormalisation equations for
  the SSM inflation tiling. Impossible distances (whose pair
  correlations evaluate to $0$ due to the tile geometry) are
  omitted. The column with $z=0$ contains the relative tile
  frequencies, as given by the right PF eigenvector of
  $M^{}_{\varrho}$. \label{tab:SilverSelf} }
\end{center}
\end{table} 

\begin{theorem}\label{thm:SilverRenorm}
  The pair correlations\/ $\nu_{\alpha \beta}$ with\/
  $\alpha, \beta \in \{a,b\}$ of the SSM tiling satisfy the exact
  renormalisation relations \allowdisplaybreaks
\begin{align*}
  %aa
  \lambda\nu^{}_{aa}(z)\, &= \,\nu^{}_{aa}\Bigl(\myfrac{z}{\lambda}\Bigr) \,+\, 
             \nu^{}_{ab}\Bigl(\myfrac{z+2}{\lambda}\Bigr) \,+\,
              \nu^{}_{ba}\Bigl(\myfrac{z-2}{\lambda}\Bigr)
  \,+\, \nu^{}_{bb}\Bigl(\myfrac{z}{\lambda}\Bigr) ,\\[2mm]
  %ab
  \lambda\nu^{}_{ab}(z)\, &= \, \nu^{}_{aa}\Bigl(\myfrac{z+2}{\lambda}\Bigr)\,+\, 
  \nu^{}_{aa}\Bigl(\myfrac{z+1}{\lambda}\Bigr)\,+\,  \nu^{}_{ab}
              \Bigl(\myfrac{z+3-\lambda}{\lambda}\Bigr)\,+\,
               \nu^{}_{ba}\Bigl(\myfrac{z}{\lambda}\Bigr)\\
     &\phantom{==}\,+\, \nu^{}_{ba}\Bigl(\myfrac{z-1}{\lambda}\Bigr)\,+\,
         \nu^{}_{bb} \Bigl(\myfrac{z+1-\lambda}{\lambda}\Bigr) ,\\[2mm]
  %ba
  \lambda\nu^{}_{ba}(z)\, &= \,
                            \nu^{}_{aa}\Bigl(\myfrac{z-2}{\lambda}\Bigr)\,+\, 
                            \nu^{}_{aa}\Bigl(\myfrac{z-1}{\lambda}\Bigr)\,+\,
                            \nu^{}_{ab}\Bigl(\myfrac{z+1}{\lambda}
  \Bigr)\,+\,  \nu^{}_{ab}\Bigl(\myfrac{z}{\lambda}\Bigr)\\
  &\phantom{==}\,+\, \nu^{}_{ba}\Bigl(\myfrac{z+\lambda-3}{\lambda}\Bigr)\,+\, 
  \nu^{}_{bb}\Bigl(\myfrac{z-1+\lambda}{\lambda}\Bigr) , \\[2mm]
  %bb 
  \lambda\nu^{}_{bb}(z)\, &= \,2   \nu^{}_{aa}
                     \Bigl(\myfrac{z}{\lambda}\Bigr)\,+\, 
                     \nu^{}_{aa}\Bigl(\myfrac{z+1}{\lambda}\Bigr)\,+\,
                     \nu^{}_{aa}\Bigl(\myfrac{z-1}{\lambda}
  \Bigr)\,+\, \nu^{}_{ab}\Bigl(\myfrac{z+1-\lambda}{\lambda}\Bigr) \\
  &\phantom{==}\,+\, \nu^{}_{ab}\Bigl(\myfrac{z+2-\lambda}{\lambda}\Bigr)\,+\, 
  \nu^{}_{ba}\Bigl(\myfrac{z+\lambda-2} {\lambda}\Bigr)\,+\,  \nu^{}_{ba}
  \Bigl(\myfrac{z-1+\lambda} {\lambda}\Bigr)\,+\, 
  \nu^{}_{bb}\Bigl(\myfrac{z+\lambda}{\lambda} \Bigr) ,
\end{align*}
with\/ $z \in \ZZ[\sqrt{2}\,]$, $\lambda= 1+\sqrt{2}$, and\/
$\nu_{\alpha \beta}(z) =0$ for\/ $z \notin \vL_{\beta}- \vL_{\alpha}$.
\qed
\end{theorem}

This is an infinite set of linear equations. However, via the
inflation structure, all arguments with $|z| > \lambda+1$ are
recursively determined from the \textit{self-consistent} part of the
equations. As this will be used later, we include the derivation. For
a similar example, where the inflation factor is non-Pisot, we refer
the reader to \cite{BFGR}.  We remark that renormalisation does not
need local recognisability, see \cite{BGMan}, but it makes the process
easier.

\begin{figure}
\vspace*{0.8cm}
\includegraphics[width = 0.85\textwidth]{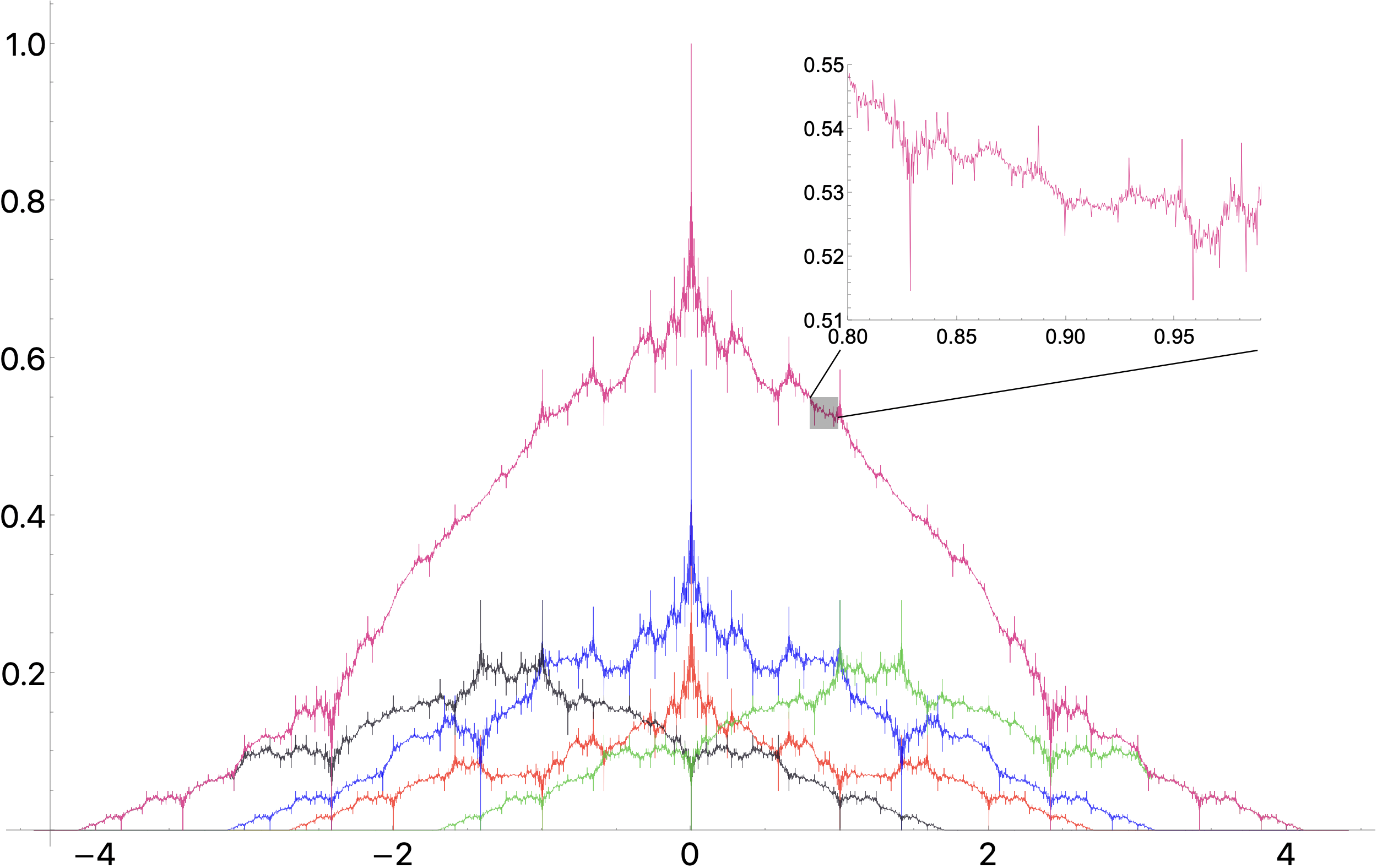}
\caption{Representation of the window covariograms of the SSM system.
  The red function is $g^{}_{aa}$, $g^{}_{bb}$ is blue, $g^{}_{ab}$ is
  black, $g^{}_{ba}$ is grey, and the covariogram $g$ of the total
  window is pink. The plot of $g$ consists of $35{\ts\ts}323$
  points; we also include a closer view of a small section to
  illustrate the irregular behaviour. \label{fig:SilverCovario} }
\end{figure}

\begin{prop}\label{prop:Silver-self-consistent}
  Consider the renormalisation relations of
  Theorem~\textnormal{\ref{thm:SilverRenorm}} with the symmetry and
  vanishing conditions, and the arguments restricted to\/
  $z \in \vL- \vL$ with\/ $|z| \leqslant \lambda+1$. This is a finite,
  closed set of equations with a one-dimensional solution space. In
  particular, once\/ $\nu^{}_{aa}(0)$ is given, the solution is
  unique.
\end{prop}

\begin{proof}
  By symmetry, we need only to look at $z\geqslant 0$, and when
  $z\leqslant \lambda+1$, no argument on the right-hand side of the
  identities in Theorem~\ref{thm:SilverRenorm} exceeds $\lambda+1$. We
  thus consider
\[
  \vL- \vL \cap [0, \lambda+1] \, = \, \{0\ts,\, 1\ts,\, 2 \ts,\,3 \ts,\,
  \lambda-1\ts,\, \lambda \ts,\, 2\lambda -2\ts,\, \lambda+1 \} \ts ,
\]
which gives the first claim. We have omitted any values prohibited by
the tile geometry, as all pair correlations would then be zero. The
values for $z=0$ follow from the PF eigenvector, and the fact that a
tile can only have one type. The geometry of the tiles implies the
other occurrences of $0$ in Table~\ref{tab:SilverSelf}; however, these
can also be checked using the relations from
Theorem~\ref{thm:SilverRenorm} and the vanishing conditions. Note that
$\lambda^{-1} = \lambda-2$, $\lambda^{2} = 1+2\lambda$, and
$\lambda^{\star} = - \lambda^{-1}$. Using these relations, we can
easily solve the resulting finite set of equations by hand, the
details of which are omitted. In particular, $\nu^{}_{aa}(0)$ is not
fixed by the relations, and all other values can be written in terms
of $\nu^{}_{aa}(0)$; thus, the solution space is indeed
one-dimensional.
   
Uniqueness follows from $\nu^{}_{aa}(0)+\nu^{}_{bb}(0)=1$, with the
values listed in Table~\ref{tab:SilverSelf}, where the interpretation
of $\nu^{}_{aa}(0)$ and $\nu^{}_{bb}(0)$ as the relative tile
frequencies was used.
\end{proof}

A similar proof was given in \cite{Klick}, to which we refer the
reader for further details and examples. Via the correspondence
between the equations from Propositions~\ref{prop:renorm} and
\ref{prop:internal}, this also gives us the renormalisation approach
to the covariogram functions $g^{}_{\alpha\beta}$.  We can now plot
the covariograms using the recursive structure of the equations, with
the self-consistent part serving as initial values. The result is
shown in Figure~\ref{fig:SilverCovario}.  Despite their appearance, by
Proposition~\ref{prop:prop_cvg}, each function is continuous and also
bounded above by a tent function. Moreover, as mentioned in
Remark~\ref{rem:rem_sample}, Figure~\ref{fig:SilverCovario} is an
acceptable illustration of the covariogram as we know its
\textit{exact} value at a dense and uniformly distributed set of
points, in combination with the continuity of the function.

\begin{figure}
\includegraphics[width=0.85 \textwidth]{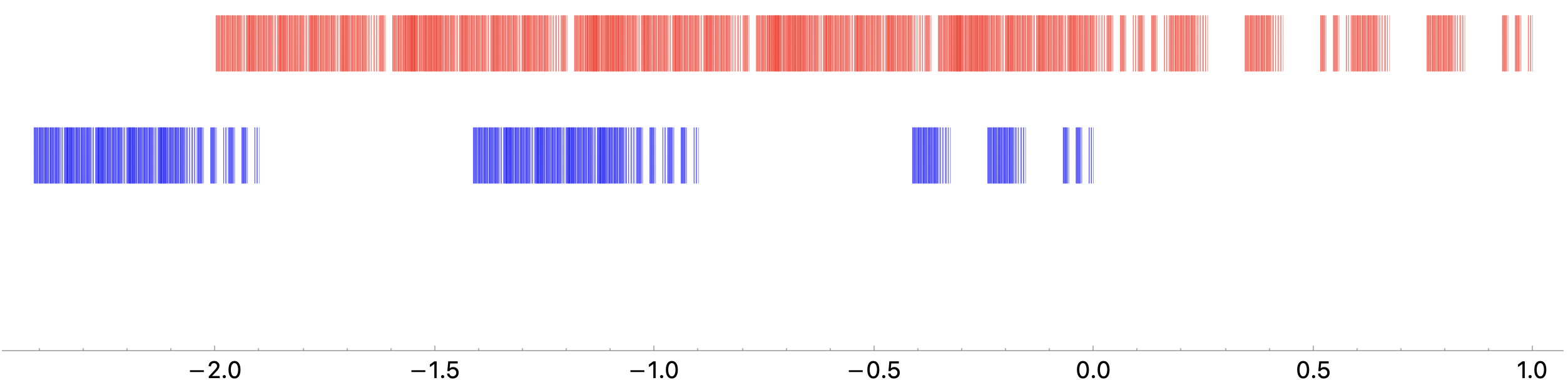}
\caption{Illustration of the windows for the substitution $\sigma$;
  $W^{}_{a}$ is red (top) and $W^{}_{b}$ is blue (bottom). As before,
  the windows are one-dimensional, but we assign some fixed arbitrary
  height to the points for illustration. Furthermore, by
  Proposition~\ref{prop:IFS_window}, the windows are
  measure-theoretically disjoint, but we are limited by the resolution
  of the plot and the large Hausdorff dimension of the window
  boundaries. \label{fig:Example-window} }
\end{figure}

\subsection{A more complex example}  

Next, we construct another window system, with the same CPS, whose
boundaries exhibit even stronger fractal behaviour than that of the
SSM system. Consider the substitution
\[
     \sigma \,:\, \begin{cases}
          a  \mapsto  aaaaabb  \ts,\\
          \ts b  \ts \mapsto  baa \ts,  \end{cases}
\]
with corresponding substitution matrix
$M^{}_{\sigma} =\big(\begin{smallmatrix} 5 & 2\\ 2 &
  1\end{smallmatrix}\big)$ and PF eigenvalue
\[
    \lambda^{}_{\mathrm{PF}} = \lambda^{2} = 2\lambda +1 =3+2\sqrt{2}\ts .
\]

The left and (statistically normalised) right PF eigenvectors are
given by
\begin{equation} \label{eq: example-PF-eigenvec}
    w^{}_{\ell} \, = \, ( \lambda, 1) \quad \text{and} \quad  
    w^{}_{r} \, = \, \Bigl(\myfrac{\lambda-1}{2}, \,\myfrac{3-\lambda}{2} 
    \Bigr)^{\!\top}. 
\end{equation}

This substitution produces a tiling of $\RR$ in an analogous way as
before. To determine the window, we go through the same procedure as
in the previous section. The relations for the sets $ \vL^{}_{a}$ and
$ \vL^{}_{b}$ are given by
\[
\begin{split} 
   \vL^{}_{a} \, & =\,\lambda^{2} \vL^{}_{a}  \; \dot{\cup}\; \big( \lambda^{2}
                   \vL^{}_{a} \,{{+}}\, \lambda \big) \; \dot{\cup}\;
                   \big(\lambda^{2} \vL^{}_{a} \,{{+}}\, 
   2 \lambda \big) \,\,\dot{\cup}\, \big(\lambda^{2}
   \vL^{}_{a} \,{{+}}\, 3 \lambda \big) \\[3pt]
          & \qquad\quad\quad \nts \dot{\cup}\,
            \big( \lambda^{2}  \vL^{}_{a} \,{{+}}\, 4 \lambda \big) 
            \; \dot{\cup}\; \big( \lambda^{2}   \vL^{}_{b} \,{{+}}\, 1 \big)
            \; \dot{\cup} \;
   \big( \lambda^{2}  \vL^{}_{b} \,{{+}}\,  \lambda \,{{+}}\,1 \big) , \\[5pt]
  \vL^{}_{b} \, & =\, \big(\lambda^{2}  \vL^{}_{a} \,{{+}}\, 5
         \lambda \big) \; \dot{\cup}\;
         \big( \lambda^{2}   \vL^{}_{a} \,{{+}}\, 5 \lambda \,{{+}}\,1 \big)
                  \; \dot{\cup}\; 
  \big( \lambda^{2}   \vL^{}_{b} \big) .
\end{split}
\]
Again, taking $W^{}_{a,b}:= \overline{ \vL^{\star}_{a,b}}$ gives the
following internal space IFS,
\[
\begin{split}
   W^{}_{a} \, & =\,\lambda^{-2} W^{}_{a}  \,\cup\, \big( \lambda^{-2}  
   W^{}_{a} \,{{-}}\, \lambda^{-1} \big) \,\cup\, \big(\lambda^{-2} 
   W^{}_{a} \,{{-}}\, 2 \lambda^{-1} \big) \,\cup\, \big(\lambda^{-2}  
   W^{}_{a} \,{{-}}\, 3 \lambda^{-1} \big) \\[3pt]
    & \qquad\quad \quad \;\;\,
    \cup\, \big( \lambda^{-2} W^{}_{a} \,{{-}}\, 4\, \lambda^{-1} \big)
      \,\cup\, \big( \lambda^{-2}  W^{}_{b} \,{{+}}\, 1 \big)
      \,\,\cup\, \big( \lambda^{-2}
      W^{}_{b} \,{{+}}\,  1\,{{-}}\,\lambda^{-1} \big) ,\\[5pt]
  W^{}_{b} \, & =\, \big(\lambda^{-2}  W^{}_{a} \,{{-}}\, 5 \lambda^{-1} \big) 
        \,\cup\,  \big( \lambda^{-2}  W^{}_{a} \,{{+}}\, 1\,{{-}}\, 5
        \lambda^{-1}  \big) \,\cup\,  \big( \lambda^{-2}  W^{}_{b} \big) .
\end{split}
\]
The fixed point of this contractive IFS is again the window system,
illustrated in Figure~\ref{fig:Example-window}.  Via the overlap
algorithm, the dimension of the fractal boundary is determined to be
\[
  \dim_{\mathrm{H}} (\partial W) \,=\, \frac{ \, \log (\gamma)}{\log
    (\lambda_{\mathrm{PF}})} \, \approx \, 0.966{\ts}4739\ts,
\]
where $\gamma \approx 5.493{\ts}959$ is the largest root of
$x^{3}-7x^{2}+7x+7$, the characteristic polynomial of the discrepancy
matrix. The renormalisation relations are as follows.

\begin{table}
\renewcommand{\arraystretch}{1.7}
\begin{tabular}{ |c|c|c|c|c|c|c|c|c|c|c|c| } 
\hline
$z$& $0$ & $1$ & $2$ & $3$& $\lambda$ & $\lambda+1$ & $\lambda+2$ & 
$\lambda+3$ & $2\,\lambda$  & $2\,\lambda+1$\\ \hline
  $\nu^{}_{aa}(z)$ & $\frac{\lambda-1}{2}$ & $0$ & $0$& $0$
    & $\frac{3\,\lambda-5}{4}$ & $\frac{17-7\,\lambda}{4}$
    & $\frac{13\,\lambda-31}{4}$ & $\frac{17-7\,\lambda}{4}$  & 
$\lambda-2$& $17-7\,\lambda$\\
  $\nu^{}_{ab}(z)$ & $0$ & $0$ & $0$& $0$ & $\frac{3-\lambda}{4}$
               & $\frac{3\,\lambda-7}{2}$ & 
                 $\frac{17-7\,\lambda}{4}$ & $0$  & $\frac{3-\lambda}{4}$
               & $\frac{3\,\lambda-7}{2}$\\
$\nu^{}_{ba}(z)$ & $0$ & $\frac{3-\lambda}{4}$ & $\frac{3\,\lambda-7}{2}$& 
                $\frac{17-7\,\lambda}{4}$ & $0$ & $\frac{3\lambda-7}{4}$
              & $\frac{3\,\lambda-7}{2}$ & 
$\frac{17-7\,\lambda}{4}$  & $0$& $\frac{3\,\lambda-7}{2}$\\
  $\nu^{}_{bb}(z)$ & $\frac{3-\lambda}{2}$ & $\frac{3-\lambda}{4}$
               & $\frac{17-7\,\lambda}{4}
$& $0$ & $0$ & $0$ & $0$ & $0$  & $0$& $\frac{17-7\,\lambda}{2}$\\ \hline
\end{tabular}
\vspace*{0.8cm}		
\caption{The self-consistent part of the renormalisation equations for
  the tiling arising from $\sigma$, with the natural tile lengths
  given by the left PF eigenvector.  Distances which are not possible,
  i.e.,~ones for which all pair correlations evaluate to zero due to
  the tile geometry, are omitted. The column corresponding to $z=0$
  contains the relative tile frequencies, as given by the
  (statistically normalised) right PF eigenvector of
  $M^{}_{\sigma}$. \label{table:example-self-consist} }
\end{table}

\begin{theorem}\label{thm:new-example-renorm}
  The pair correlations\/ $\nu_{\alpha \beta}$ with\/
  $\alpha, \beta \in \{a,b\}$ of the tiling induced by\/ $\sigma$
  satisfy the exact renormalisation relations \allowdisplaybreaks[4]
\begin{align*}
 %aa%%%%%%%%%%%%%%%%%%%%%%%
  \lambda^{2} \nu^{}_{aa}(z)\, &=
        \, 5 \nu^{}_{aa}\Bigl(\myfrac{z}{\lambda^{2}}\Bigr)  +  
   4 \nu^{}_{aa}\Bigl(\myfrac{z-\lambda}{\lambda^{2}}\Bigr)   +  
   4 \nu^{}_{aa}\Bigl(\myfrac{z+\lambda}{\lambda^{2}}\Bigr)   +   
   3 \nu^{}_{aa}\Bigl(\myfrac{z-2\,\lambda}{\lambda^{2}}\Bigr)  + 
   \nu^{}_{bb}\Bigl(\myfrac{z+\lambda}{\lambda^{2}}\Bigr)\\
         &\phantom{==}  +  3 \nu^{}_{aa}
           \Bigl(\myfrac{z+2\lambda}{\lambda^{2}}\Bigr)   +  
   2 \nu^{}_{aa}\Bigl(\myfrac{z-3\,\lambda}{\lambda^{2}}\Bigr)   +   
   2 \nu^{}_{aa}\Bigl(\myfrac{z+3\,\lambda}{\lambda^{2}}\Bigr)   +  
   \nu^{}_{aa}\Bigl(\myfrac{z-4\,\lambda}{\lambda^{2}}\Bigr)\\
       &\phantom{==}  +  \nu^{}_{aa}
          \Bigl(\myfrac{z+4\,\lambda}{\lambda^{2}}\Bigr)  +  
   2 \nu^{}_{ab}\Bigl(\myfrac{z-1}{\lambda^{2}}\Bigr)   +  2
   \nu^{}_{ba}\Bigl(\myfrac{z+1}{\lambda^{2}}\Bigr)  +  
   \nu^{}_{ab}\Bigl(\myfrac{z-\lambda-1}{\lambda^{2}}\Bigr)\\
   &\phantom{==}  + \nu^{}_{ba}\Bigl(\myfrac{z+\lambda+1}{\lambda^{2}}\Bigr)   
   +  \nu^{}_{ab}\Bigl(\myfrac{z+4\,\lambda-1}{\lambda^{2}}\Bigr)   +  
   \nu^{}_{ba}\Bigl(\myfrac{z-4\,\lambda+1}{\lambda^{2}}\Bigr) + \nu^{}_{bb}
   \Bigl(\myfrac{z-\lambda}{\lambda^{2}}\Bigr) \\
   &\phantom{==}  +
        2 \nu^{}_{ab}\Bigl(\myfrac{z+\lambda-1}{\lambda^{2}}\Bigr)  
   +  2\nu^{}_{ba}\Bigl(\myfrac{z-\lambda+1}{\lambda^{2}}\Bigr)  + 2\nu^{}_{ab}
   \Bigl(\myfrac{z+2\lambda-1}{\lambda^{2}}\Bigr)  + 2 \nu^{}_{bb}
   \Bigl(\myfrac{z}{\lambda^{2}}\Bigr)\\
      &\phantom{==}  +  2 \nu^{}_{ba}
       \Bigl(\myfrac{z-2\,\lambda+1}{\lambda^{2}}\Bigr)  
   + 2\nu^{}_{ab}\Bigl(\myfrac{z+3\,\lambda-1}{\lambda^{2}}\Bigr)  + 
   2\nu^{}_{ba}\Bigl(\myfrac{z-3\,\lambda+1}{\lambda^{2}}\Bigr)\ts,\\[2mm]
 %bb  
  \lambda^{2} \nu^{}_{bb}(z)\, &=  \,2  \nu^{}_{aa}
               \Bigl(\myfrac{z}{\lambda^{2}}\Bigr)  
          +  \nu^{}_{aa}\Bigl(\myfrac{z-1}{\lambda^{2}}\Bigr)  +
        \nu^{}_{ab}\Bigl(\myfrac{z+5\, \lambda}
   {\lambda^{2}}\Bigr)  +  \nu^{}_{aa}\Bigl(\myfrac{z+1}{\lambda^{2}}\Bigr) \\
         &\phantom{==}  +  \nu^{}_{ab}\Bigl(
          \myfrac{z+5\, \lambda^{}+1}{\lambda^{2}}\Bigr) 
   +  \nu^{}_{ba}\Bigl(\myfrac{z-5\,\lambda^{}} {\lambda^{2}}\Bigr)  +   
   \nu^{}_{ba}\Bigl(\myfrac{z-5\,\lambda^{}-1} {\lambda^{2}}\Bigr)  +  
   \nu^{}_{bb}\Bigl(\myfrac{z}{\lambda^{2}}\Bigr), \\[2mm]
 %ab
  \lambda^{2}\nu^{}_{ab}(z)\, &=
          \,\nu^{}_{aa}\Bigl(\myfrac{z-5\lambda}{\lambda^{2}}\Bigr) 
    + \nu^{}_{aa}\Bigl(\myfrac{z-5\lambda-1}{\lambda^{2}}\Bigr)   +  
    \nu^{}_{ab}\Bigl(\myfrac{z}{\lambda^{2}}\Bigr)  + \nu^{}_{aa}
    \Bigl(\myfrac{z-4\lambda}{\lambda^{2}}\Bigr)  +  \nu^{}_{ba}
    \Bigl(\myfrac{z-4\lambda} {\lambda^{2}}\Bigr)\\
         &\phantom{==}  +  \nu^{}_{aa}
          \Bigl(\myfrac{z-4\lambda-1} {\lambda^{2}} \Bigr)   
    +  \nu^{}_{ab}\Bigl(\myfrac{z+\lambda}{\lambda^{2}}\Bigr)   +  \nu^{}_{aa}   
    \Bigl(\myfrac{z-3\lambda}{\lambda^{2}}\Bigr)   +  
    \nu^{}_{aa}\Bigl(\myfrac{z-3\lambda-1} {\lambda^{2}}\Bigr) \\
         &\phantom{==}  +  \nu^{}_{ab}
           \Bigl(\myfrac{z+2\lambda}{\lambda^{2}}\Bigr)   +  
    \nu^{}_{aa}\Bigl(\myfrac{z-2\lambda}{\lambda^{2}}\Bigr)   +  \nu^{}_{aa}
    \Bigl(\myfrac{z-2\lambda-1}{\lambda^{2}}\Bigr)   +  
    \nu^{}_{ab}\Bigl(\myfrac{z+3\lambda}{\lambda^{2}}\Bigr) \\
        &\phantom{==}  +  \nu^{}_{aa}
          \Bigl(\myfrac{z-\lambda}{\lambda^{2}}\Bigr)   +  
    \nu^{}_{aa}\Bigl(\myfrac{z-\lambda-1} {\lambda^{2}}\Bigr)  +  
    \nu^{}_{ab}\Bigl(\myfrac{z+4\lambda} {\lambda^{2}}\Bigr)  +  \nu^{}_{ba}  
    \Bigl(\myfrac{z+1-5\lambda}{\lambda^{2}}\Bigr)\\
        &\phantom{==}  +  \nu^{}_{ba}
         \Bigl(\myfrac{z-5\lambda} {\lambda^{2}}\Bigr)  + 
     \nu^{}_{bb}\Bigl(\myfrac{z+1} {\lambda^{2}}\Bigr)  +  \nu^{}_{ba}
     \Bigl(\myfrac{z+1-4\lambda}{\lambda^{2}}\Bigr)   +  
     \nu^{}_{bb}\Bigl(\myfrac{z+\lambda+1} {\lambda^{2}}\Bigr) ,
\end{align*}
together with\/ $\nu^{}_{ba}(z) = \nu^{}_{ab}(-z)$, $z \in \ZZ[\sqrt{2}\,]$, 
$\lambda = 1+ \sqrt{2}$, and\/ $\nu_{\alpha \beta}(z) =0$ for\/
$z \notin  \vL_{\beta}- \vL_{\alpha}$.  \qed
\end{theorem}

The self-consistent part of the relations is listed in
Table~\ref{table:example-self-consist}. As the calculations are
analogous to Proposition~\ref{prop:Silver-self-consistent}, we omit a
proof.

\begin{figure}
\vspace*{0.8cm}
\includegraphics[width=0.85\textwidth]{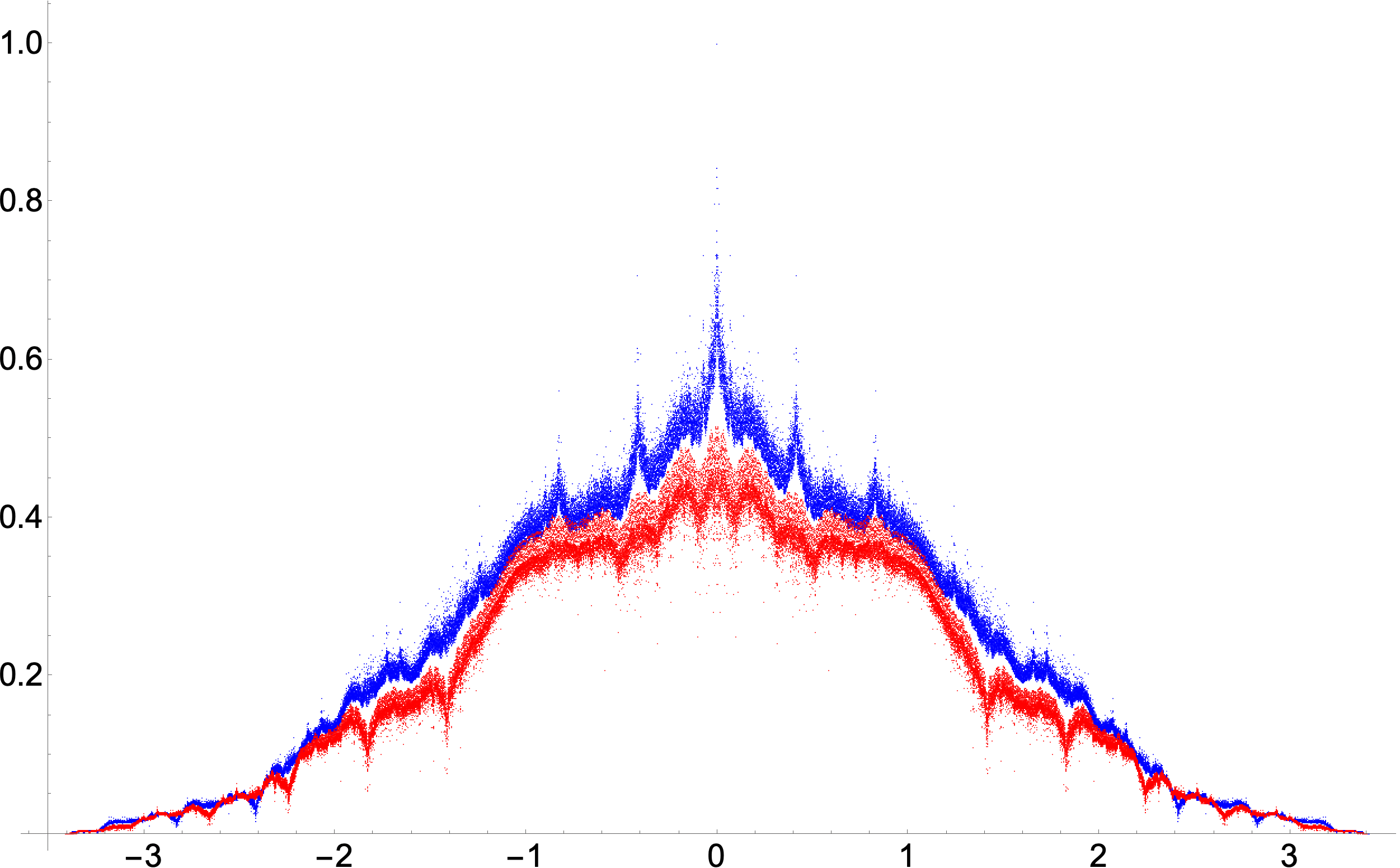}
\vspace*{0.8cm}
\caption{Point plot, with $188{\ts\ts}214$ points, of the covariogram
  of the window of Figure~\ref{fig:Example-window}. Here, the
  splitting behaviour is highlighted; the distances involving an even
  number of $b$'s are blue, and those with an odd number are
  red.  \label{fig:ExampleCovarioSplit}}
\end{figure}

\begin{coro}
  Consider the renormalisation relations of
  Theorem~\textnormal{\ref{thm:new-example-renorm}} with the symmetry
  and vanishing conditions, and the arguments restricted to\/
  $z \in \vL- \vL$ with\/ $|z| \leqslant 2\,\lambda+1$. This is a
  finite, closed set of equations with a one-dimensional solution
  space. In particular, once\/ $\nu^{}_{aa}(0)$ is given, the solution
  is unique.  \qed
\end{coro}    
 
The covariogram can now be computed as before, and is shown in
Figures~\ref{fig:ExampleCovarioSplit} and
\ref{fig:ExampleCovarioBig}. We remark again that this is indeed a
continuous function, despite its \emph{highly} discontinuous
appearance. To our knowledge, this is the first example of such a
covariogram having a `splitting' behaviour.  One can do the following
as a consistency check for the continuity. Take a large patch of the
tiling, approximately $500{\ts\ts}000$ tiles, and construct the set of
tile coordinates.  The Minkowski difference of this set is then taken
with itself; with the above number of tiles, it should produce
approximately five million distinct distances. To efficiently run this
set of distances through the relations from
Theorem~\ref{thm:new-example-renorm}, one can implement a `valid
distance' test; this checks whether the integer coordinates of
$z= a\ts\lambda + b$ have the same parity (or are zero) and
whether~$z^{\star}$ lies in a bounded region of internal space.  The
distance $z$ is then passed to the relations only if both requirements
are met.  Additionally, running the relations in parallel
significantly improves computation time.

\begin{figure}
\includegraphics[width=0.85 \textwidth]{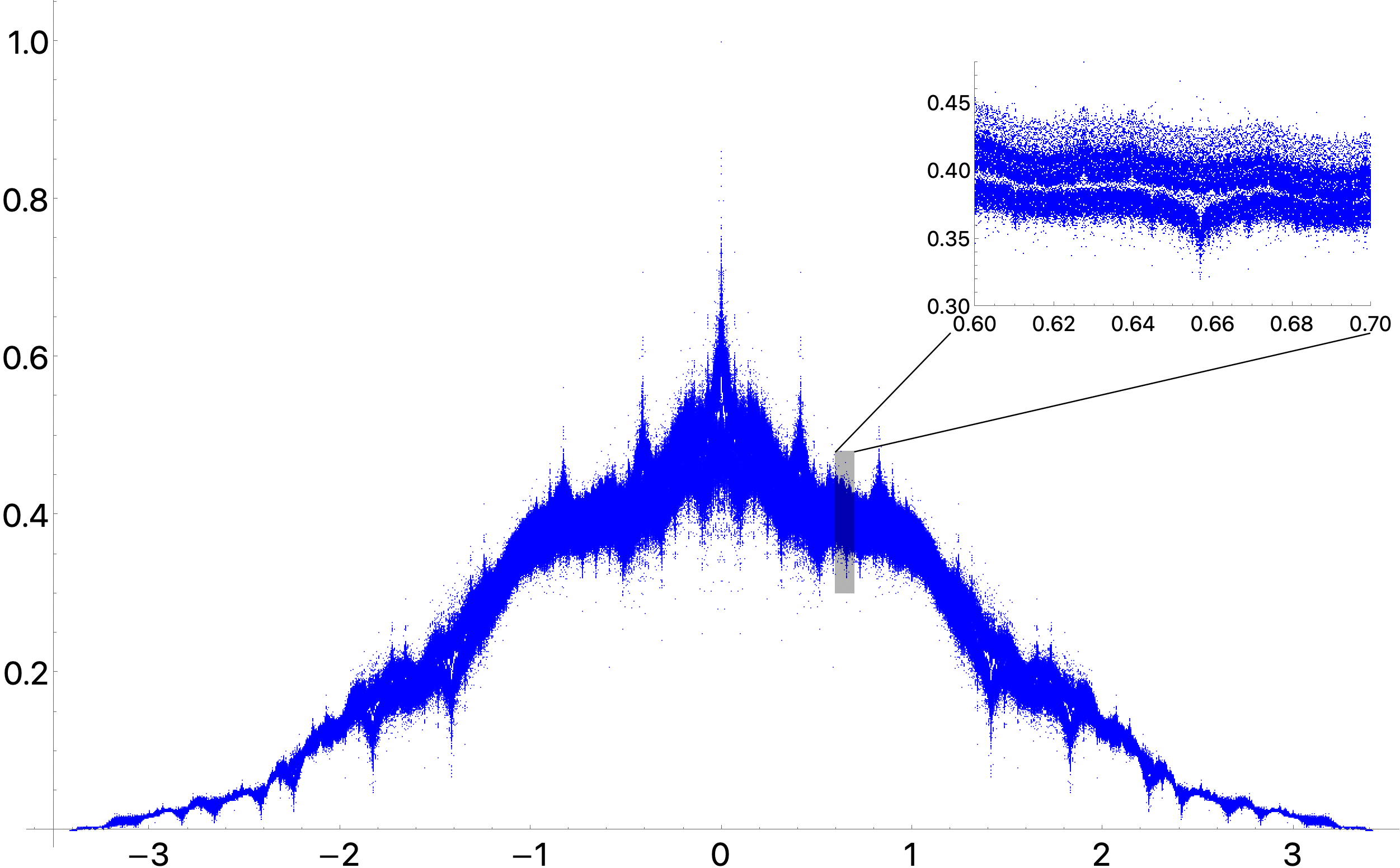}
\caption{Point plot, derived from $4{\ts\ts}826{\ts\ts}361$ points, of
  the covariogram of the window of Fig.~\ref{fig:Example-window}. The
  closer view illustrates the slow convergence and complex behaviour
  of the function.  \label{fig:ExampleCovarioBig}}
\end{figure}  

With these large samples, one sees that the gap will close, albeit
incredibly slowly; Figure~\ref{fig:ExampleCovarioBig} illustrates this
behaviour. We emphasise again, as in Remark~\ref{rem:rem_sample} and
in the previous example, that this is a valid approximation of the
continuous function as we are sampling from a dense and uniformly
distributed subset of the support. The upper and lower branches of the
split can be identified as those distances which contain,
respectively, an even and odd number of $b$'s. This behaviour is thus
a consequence of the \emph{combinatorics} of the substitution and not
triggered by the high boundary dimension of the window, as one might
first guess. This is well worth further examination, but is far beyond
the scope of this brief note. Also, with regard to the original
question of a natural counterpart to the theta series of periodic
point sets, the complex behaviour of the covariograms will also call
for a renewed, and perhaps conceptually different, approach to
aperiodic theta series.

Our above analysis was based on one-dimensional inflation tilings with
an internal space that is also one-dimensional. It is an obvious task
to extend both to higher dimensions. In a first step, one should
consider unimodular inflation tilings with a ternary alphabet, thus
giving Rauzy fractals in $\RR^2$. First results on the Kolakoski
sequence are presented in \cite{Klick}, and the same type of analysis
can be done for the classic Tribonacci sequence. Computationally, both
are significantly more involved yet of limited value because the
windows are simply connected sets with a fractal boundary. The
resulting covariograms still look well behaved and mainly smooth. More
interesting would be the \emph{twisted} Tribonacci sequence
\cite{AI01,BGG-Orbit}, which leads to a disconnected window system
that may be considered as an analogue of a Cantorval.  More work is
needed to explore the possibilities here.

\section*{Acknowledgements}

It is our pleasure to thank Claudia Alfes and Paul Kiefer for helpful
discussions, and Michael Coons and Nicolae Strungaru for useful hints
on the manuscript. We are grateful to an anonymous referee for several
careful and constructive comments, which helped us to improve the
presentation. M.B. is grateful to the School of Mathematics and
Physics of the University of Tasmania in Hobart for hospitality, where
this manuscript was completed.  This work was supported by the German
Research Foundation (DFG, Deutsche Forschungsgemeinschaft) via
Project~A1 within the CRC TRR 358/1 (2023) -- 491392403 (Bielefeld --
Paderborn).


\begin{thebibliography}{99}
\itemsep=0.5pt

\bibitem{AKM2024}
C.~Alfes, P.~Kiefer and J.~Maz{\'a}{\v c}, 
Measures, modular forms, and summation formulas of Poisson type, 
\textit{Commun.\ Math.\ Phys.} \textbf{406} (2025),
art.~137 (23 pp); \texttt{arXiv:2405.15620}.

\bibitem{APisot}
S.~Akiyama, M.~Barge, V.~Berth{\'e}, J-Y.~Lee and A.~Siegel, 
On the Pisot substitution conjecture, in:
\textit{Mathematics of Aperiodic Order}, 
eds. J.~Kellendonk, D.~Lenz and J.~Savinien,  
Birkh{\"a}user, Basel (2015), pp. 33--72.

\bibitem{AOverlap}
S.~Akiyama, and J-Y.~Lee, 
Algorithm for determining pure pointedness of self-affine tilings,\
\textit{Adv.\ Math.} \textbf{226} (2011), 2855--2883;
\texttt{arXiv:1003.2898}. 

\bibitem{AHRauzy}
P.~Arnoux and E.~Harriss, 
What is a Rauzy fractal, 
\textit{Notices Amer.\ Math.\ Soc.} \textbf{61} (2014), 768--770.

\bibitem{AI01} 
P.~Arnoux and S.~Ito, 
Pisot Substitutions and Rauzy fractals,
\textit{Bull.\ Belg.\ Math.\ Soc.} \textbf{8} (2001), 181--207.

\bibitem{BFGR}
M.~Baake, N.~P.~Frank, U.~Grimm and E.~A.~Robinson, 
Geometric properties of a binary non-Pisot inflation and absence of
absolutely continuous diffraction, 
\textit{Studia Math.} \textbf{247} (2019), 109--154;
\texttt{arXiv:1706.03976}.

\bibitem{BG16}
M.~Baake and  F.~G\"{a}hler,
Pair correlations of aperiodic inflation rules via renormalisation:
Some interesting examples,
\textit{Topol.\ Appl.} \textbf{205} (2016), 4--27; 
\texttt{arXiv:1511.00885}.

\bibitem{BGG-Orbit}
M.~Baake, F.~G{\"a}hler and P.~Gohlke, 
Orbit separation dimension as complexity measure for primitive
inflation tilings,
\textit{Ergod.\ Th.\ Dynam.\ Syst.} \textbf{45} (2025), 2992--3020;
\texttt{arXiv:2311.03541}.

\bibitem{BGMan}
M.~Baake,  F.~G{\"a}hler and N.~Ma{\~n}ibo, 
Renormalisation of pair correlation measures for primitive inflation rules and 
absence of absolutely continuous diffraction, 
\textit{Commun.\ Math.\ Phys.} \textbf{370} (2019), 591--635; 
\texttt{arXiv:1805.09650}.

\bibitem{BGM2024}
M.~Baake, A.~Gorodetski and J.~Maz{\'a}{\v c}, 
A naturally appearing family of Cantorvals, 
\textit{Lett.\ Math.\ Phys.} \textbf{114} (2024), art.~101 (11 pp); 
\texttt{arXiv:2401.05372}.

\bibitem{NOS}
M.~Baake and U.~Grimm, 
A note on shelling, 
\textit{Discr.\ Comput.\ Geom.} \textbf{30} (2003), 573--589; \newline
\texttt{arXiv:math/0203025}.

\bibitem{BG-Homometric}
M.~Baake and U.~Grimm, 
Homometric model sets and window covariograms, 
\textit{Z.\ Krist.} \textbf{222} (2007), 54--58; 
\texttt{arXiv:math/0610411}.

\bibitem{TAO}
M.~Baake and U.~Grimm,
\textit{Aperiodic Order. Vol.\ 1: A Mathematical Invitation},
Cambridge University Press, Cambridge (2013).
  
\bibitem{BG2020}
M.~Baake and U.~Grimm, 
Fourier transform of Rauzy fractals and point spectrum of
1D Pisot inflation tilings, 
\textit{Docum.\ Math}. \textbf{25} (2020), 2303–2337;
\texttt{arXiv:1907.11012}.

\bibitem{BM-Self}
M.~Baake and R.~V.~Moody, 
Self-similar measures for quasicrystals, 
in: \textit{Directions in Mathematical Quasicrystals}, 
eds. M.~Baake and R.~V.~Moody,  CRM Monograph Series, 
vol.~13, AMS, Providence RI, 2000, pp. 1--42.

\bibitem{BalRus}
S.~Balchin and D.~Rust, 
Computations for symbolic substitutions, 
\textit{J. Int. Seq}. \textbf{20} (2017), art.~17.4.1 (36pp); 
\texttt{arXiv:1705.11130}. 

\bibitem{Barge}
M.~Barge, 
Pure discrete spectrum for a class of one-dimensional substitution
tiling systems,
\textit{Discr.\ Cont.\ Dynam.\ Syst. A} \textbf{36} (2016), 1159--1173; 
\texttt{arXiv:1403.7826}.

\bibitem{BK}
M.~Barge and J.~Kwapisz,
Geometric theory of of unimodular Pisot substitutions,
\textit{Amer.\ J.\ Math.} \textbf{128} (2006), 1219--1282.

\bibitem{BST}
V.~Berth{\'e}, A.~Siegel and J.~Thuswaldner, 
Substitutions, Rauzy fractals and tilings, 
in: \textit{Combinatorics, Automata and Number Theory}, ed.\ V.~Berth{\'e}, 
Cambridge Univ. Press, Cambridge (2010), pp. 248--323. 

\bibitem{Bianchi}
G.~Bianchi, 
The covariogram problem, 
in: \textit{Harmonic Analysis and Convexity}, 
eds. A.~Koldobsky and A.~Volbergs, 
de Gruyter, Berlin (2023), pp. 37--82.

\bibitem{BHK}
M.~Bj\"{o}rklund, T.~Hartnick, Y.~Karasik, 
Intersection spaces and multiple transverse recurrence, 
\textit{J.\ Anal.\ Math.} \textbf{156} (2025), 97--150;
\texttt{arXiv:2108.09064}.

\bibitem{ConSl}
J.~H.~Conway and N.~J.~A.~Sloane, 
\textit{Sphere Packings, Lattices and Groups}, 
3rd ed., Springer, New York, (1999).

\bibitem{HS2003}
M.~Hollander and B.~Solomyak, 
Two-symbol Pisot substitutions have pure discrete spectrum, 
\textit{Ergod.\ Th.\ Dynam.\ Syst.} \textbf{23} (2003), 533--540.

\bibitem{Hutch}
J.~E.~Hutchinson, 
Fractals and self-similarity, 
\textit{Indiana Univ.\ Math.\ J.} \textbf{30} (1981), 713--747.

\bibitem{IR06} 
S.~Ito and H.~Rao, 
Atomic surfaces, tilings and coincidence I: irreducible case, 
\textit{Israel J.\ Math.} \textbf{153} (2006), 129--155. 

\bibitem{Klick}
A.~Klick, 
\textit{Averaged Shelling of Model Sets via Renormalisation}, 
Master thesis, Univ. Bielefeld (2024).

\bibitem{Koch}
H.~Koch, 
\textit{Number Theory: Algebraic Numbers and Functions}, 
Amer.\ Math.\ Soc., Providence, RI (2000).

\bibitem{KN}
L.~Kuipers and H.~Niederreiter, 
\textit{Uniform Distribution of Sequences}, reprint, Dover,
New York (2006).

\bibitem{ManiboThesis}
N.~Ma\~{n}ibo, 
\textit{Lyapunov Exponents in the Spectral Theory of Primitive
  Inflation Systems}, 
PhD thesis, Univ.\ Bielefeld (2019), \newline
available electronically at 
\texttt{urn:nbn:de:0070-pub-29359727}.

\bibitem{Jan}
J.~Maz\'{a}\v{c},
\textit{Fractal and Statistical Phenomena in Aperiodic Order},
PhD thesis, Univ.\ Bielefeld (2025), available electronically at
\texttt{urn:nbn:de:0070-pub-30062996}.

\bibitem{Maz25}
J.~Maz\'{a}\v{c}, 
Exact renormalisation for patch frequencies in inflation systems,
\textit{preprint} (2025); \newline 
\texttt{arXiv:2507.07753}.

\bibitem{Meyer}
Y.~Meyer, 
\textit{Algebraic Numbers and Harmonic Analysis},
North Holland, Amsterdam (1972).

\bibitem{Moo97}
R.~V.~Moody, 
Meyer sets and their duals, in:
\textit{The Mathematics of Long-Range Aperiodic Order}, 
ed. R.~V.~Moody, Kluwer, Dordrecht (1997), pp. 403--441.

\bibitem{MW}
R.~V.~Moody and A.~Weiss, 
On shelling $E_{8}$ quasicrystals, 
\textit{J.~Number~Th.} \textbf{47} (1994), 405--412.

\bibitem{STTopo}
A.~Siegel and J.~Thuswaldner, 
\textit{Topological Properties of Rauzy Fractals}, 
Soci\'{e}te Math\'{e}matique de France, Paris (2009).
 
\bibitem{SingThesis}
B.~Sing, 
\textit{Pisot Substitutions and Beyond}, 
PhD thesis, Univ.\ Bielefeld (2007),  available electronically at 
\texttt{urn:nbn:de:hbz:361-11555}.

\bibitem{SW02}
V.~F.~Sirvent and Y.~Wang, 
Self-affine tiling via substitution dynamical systems and Rauzy fractals,
\textit{Pacific J.\ Math.} \textbf{206} (2002), 465--485.

\bibitem{SolSelfSimilar}
B.~Solomyak, 
Dynamics of self-similar tilings, 
\textit{Ergod.\ Th.\ Dynam.\ Syst.} \textbf{17} (1997), 695--738, and 
\textit{Ergod.\ Th.\ Dynam.\ Syst.} \textbf{19} (1999), 1685--1685 (erratum).

\end{thebibliography}
\end{document}